\documentclass[12pt, a4paper,leqno,oneside]{amschanged}

\usepackage{hyperref}
\usepackage{amsmath}
\usepackage{amssymb}
\usepackage{amsthm}
\usepackage{atbeginend}
\usepackage{psfrag}
\usepackage[utf8]{inputenc}
\usepackage{graphicx}
\usepackage{hyperref}
\usepackage{tikz}
\usetikzlibrary{decorations.pathmorphing, patterns, calc, snakes}
\hypersetup{
    colorlinks,%
    citecolor=black,%
    filecolor=black,%
    linkcolor=black,%
    urlcolor=black
}

\usepackage{fullpage}
\newtheorem{theorem}{Theorem}[section]

\newtheorem{definition}[theorem]{Definition}
\newtheorem{proposition}[theorem]{Proposition}
\newtheorem{lemma}[theorem]{Lemma}

\numberwithin{equation}{section}

\theoremstyle{remark}
\newtheorem{remark}[theorem]{Remark}

\usepackage{titlesec}
\titleformat{\section}{\Large\bfseries}{\thesection}{1em}{}
\titleformat{\subsection}{\bfseries}{\thesubsection}{1em}{}

\renewcommand{\theenumi}{\arabic{enumi}}
\renewcommand{\labelenumi}{\theenumi)}

\usepackage{courier}

\usepackage{array}
\newcolumntype{e}{>{\displaystyle}r @{\,} >{\displaystyle}c @{\,} >{\displaystyle}l}

\newcounter{constant}

\begin{document}

\fontsize{12}{14}\rm
\addtolength{\abovedisplayskip}{.5mm}
\addtolength{\belowdisplayskip}{.5mm}
\AfterBegin{enumerate}{\addtolength{\itemsep}{2mm}}

\title{\LARGE \usefont{T1}{tnr}{b}{n} \selectfont B\lowercase{ernoulli line percolation}} 

\author{\normalsize \itshape M.R. H\lowercase{il\'ario} $^{1,2}$ \color{white} \tiny and}
\address{$^1$ Universidade Federal de Minas Gerais, Departamento de Matem\'atica, 
\newline \hspace*{10mm} Belo Horizonte 31270-901, Brazil 
{\itshape \texttt{mhilario@mat.ufmg.br}}.}
\address{$^2$ Université de Genève, Department de Mathématiques,
\newline \hspace*{10mm} Rue du Lièvre 2-4, 1211 Genève, Switzerland.}

\author{\color{black} \normalsize \itshape V. S\lowercase{idoravicius} $^{3,4}$ \color{white} \tiny}

\address{$^3$ Courant Institue of Mathematical Sciences, NYU,
\newline \hspace*{10mm} 251 Mercer Street, New York, NY 10012, USA, and}
\medskip
\address{$\; \, $ NYU-Shanghai,
\newline \hspace*{10mm} 1555 Century Av., Pudong Shanghai, CN 200122, China.}

\address{$^4$ Instituto Nacional de Matem\'atica Pura e Aplicada, \newline \hspace*{10mm} 
Rio de Janeiro 22460-320, Brazil
{\itshape \texttt{v.sidoravicius@gmail.com}}.}
\date{\today}

\begin{abstract}
We introduce a percolation model on $\mathbb{Z}^d$, $d \geq 3$, in which the discrete lines of vertices that are parallel to the coordinate axis are entirely removed at random and independently of each other.
In this way a vertex belongs to the vacant set $\mathcal{V}$ if and only if none of the $d$  lines to which it belongs, is removed.
We show the existence of a phase transition for $\mathcal{V}$ as the probability of removing the lines is varied. 
We also establish that, in the certain region of parameters space where $\mathcal{V}$ contains an infinite component, the truncated connectivity function has power-law decay, while inside the region where $\mathcal{V}$ has no infinite component, there is a transition from exponential to power-law decay.
In the particular case $d=3$ the power-law decay extends through all the region where $\mathcal{V}$ has an infinite connected component.
We also show that the number of infinite connected components of $\mathcal{V}$ is either $0$, $1$ or $\infty$.
 
\end{abstract}

\maketitle

\section{Introduction}
\label{sec:intr}

In the classical Bernoulli site or bond percolation, the vertices (sites) or edges (bonds) of an infinite connected graph are  removed independently with a given probability $1-p$.
Geometric properties of the resulting random graph such as the number of infinite connected components and the decay of connectivity as a function of the parameter $p$ are among central questions of 
Percolation Theory.
In this article we address the same set of questions for a percolation model on $\mathbb{Z}^d$ presenting strong directional correlations.

To set the idea informally, consider a large cube made of a solid material and select at random  points on each of its faces.
For each selected point, drill holes through it, perpendicularly to the face it belongs to, and traversing the cube straight to the opposite face.
Whether the cube remains essentially a well-connected solid object or it is shattered into many tiny pieces 
depends on how many and in which way the points on each face are selected.

In this work we introduce a percolation model for which the state of each vertex is determined by random variables associated to its projections into $(d-1)$-dimensional subspaces.
This implies that the states of vertices aligned to each other in the lattice are strongly correlated and, in fact, the correlations do not vanish as a function of the distance between them.
The system undergoes a phase transition for existence of an infinite connected component as we vary some control parameters, however, when the truncated connectivity function 
or the radius of the cluster containing the origin is considered, 
some interesting phenomena arise. 
Namely,  there are multiple transitions in the connectivity decay different from what is obtained for ``classical" nearest neighbor or some short-range percolation models.
In particular, there is a transition from exponential to power-law decay inside the subcritical phase.

Let us now define the model and state the main results precisely. 
Consider $\mathbb{Z}^d$,  $d \geq 3,$  denote by $\{ e_1, \ldots, e_d\}$ a canonical 
orthonormal basis and by  $\mathcal{P}_i \subset \mathbb{Z}^d$ the subspace orthogonal to $e_i$, \emph{i.e.,} the set of points of $\mathbb{Z}^d$ having the $i$-th coordinate equal to $0$. 
For each $i = 1,\ldots,d$, and $w \in \mathcal{P}_i$ let 
$\ell_i(w):=\{w + z e_i,  ~ z \in \mathbb{Z} \}$  stand for the discrete line parallel to $e_i$ 
that contains the vertex $w$. 

Fix $d$ parameters $p_1, \ldots, p_d$ in the interval $[0,1]$.
For each $i$ and each $w \in \mathcal{P}_i$, we remove independently with probability 
$1-p_i$ the entire line $\ell_i(w)$,
\textit{i.e.} we remove at once all the vertices lying along $\ell_i(w)$.
We are then left with a random subset $\mathcal{V} \subset \mathbb{Z}^d$ consisting 
of the vertices that have not been removed. In analogy with the terminology used in random interlacements,
$\mathcal{V}$ will be called the \textit{vacant set}, and 
the study of its connectivity properties is the main goal of this article. 

We write $\mathbf{p} = (p_1,\ldots, p_d)$ and denote by $\mathbb{P}_\mathbf{p}$ the corresponding 
probability law of this process.
Let $\{\mathbf{0} \leftrightarrow \infty\}$ be the event that there exists an infinite path of 
nearest-neighbor vertices in $\mathcal{V}$ starting at the origin $\mathbf{0}\in\mathbb{Z}^d$ or, equivalently, that the origin belongs to an infinite connected component of $\mathcal{V}$.
For a given $\mathbf{p}$, we say that $\mathcal{V}$ percolates if $\mathbb{P}_{\mathbf{p}} (\mathbf{0} \leftrightarrow \infty) > 0$ which, by ergodicity, implies that 
$\mathcal{V}$ has at least one infinite connected component almost surely.
Let $B(n) := [ -n, n]^d \cap \mathbb{Z}^d$ be the box of side-length $2n$, centered at $\mathbf{0}$ and  
denote by $\{\mathbf{0} \leftrightarrow \partial{B}(n)\}$  the event that there exists a 
nearest neighbor path in $\mathcal{V}$ connecting the origin to the boundary of $B(n)$. 

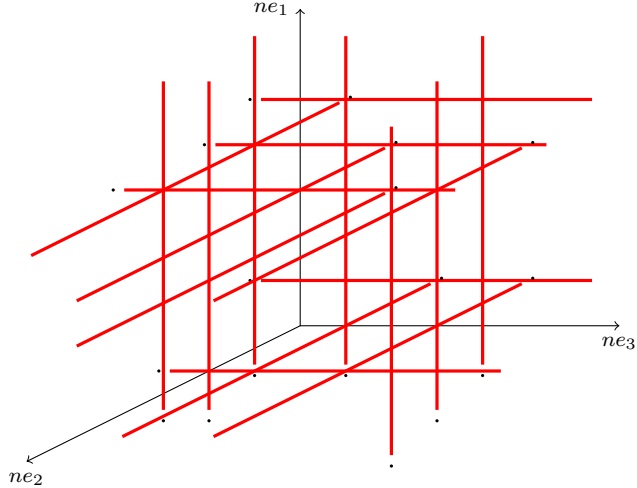
\begin{figure}
\begin{tikzpicture}[scale=.6]
\draw[<->] (0,7) -- (0,0) -- (7,0);
\draw[->] (0,0) -- (-6,-3);
\node[left] at (0,7) {\tiny$ne_1$};
\node[below] at (-6,-3) {\tiny$ne_2$};
\node[below] at (7,0) {\tiny$ne_3$};

\fill (-1,-1.1) circle [radius=1pt] node (z1) {}
      (1,-1.1) circle [radius=1pt] node (z2) {}
      (4,-1.1) circle [radius=1pt] node (z3) {}
      (-2,-2.1) circle [radius=1pt] node (z4) {}
      (-3,-2.1) circle [radius=1pt] node (z5) {}
      (3,-2.1) circle [radius=1pt] node (z6) {}
      (2,-3.1) circle [radius=1pt] node (z7) {};

\foreach \i in {1,...,7}
 {
  \draw[red, very thick] (z\i) -- ($(z\i)+(0,7.5)$);
 }

\fill (-1.1,1) circle [radius=1pt] node (y1) {}
     (-3.1,-1) circle [radius=1pt] node (y2) {}
     (-4.1,3) circle [radius=1pt] node (y3) {}
     (-1.1,5) circle [radius=1pt] node (y4) {}
     (-2.1,4) circle [radius=1pt] node (y5) {};

\foreach \i in {1,...,5}
 {
  \draw[very thick, red] (y\i) -- ($(y\i)+(7.5,0)$);
 }

\fill ($(3,1)+(.1,.05)$) circle [radius=1pt] node (x1) {}
     (1,5)+(.1,.05) circle [radius=1pt] node (x2) {}
     (2,3)+(.1,.05) circle [radius=1pt] node (x3) {}
     (5,1)+(.1,.05) circle [radius=1pt] node (x4) {}
     (2,4)+(.1,.05) circle [radius=1pt] node (x5) {}
     (5,4)+(.1,.05) circle [radius=1pt] node (x6) {};
\foreach \i in {1,...,6}
 {
  \draw[very thick, red] (x\i) -- ($(x\i)+(-7,-3.5)$);
 }
\end{tikzpicture}
\caption{A random set of lines when restricted to a quadrant of $\mathbb{Z}^3$. The vacant set $\mathcal{V}$ is the set of vertices of $\mathbb{Z}^3$ that are not covered by any of the lines.}
\end{figure}

When each given line is removed with high probability
(\emph{i.e.\ }when all the parameters $p_i$ are small) it is natural to expect that the procedure will disconnect $\mathbb{Z}^d$ into infinitely many finite components.
On the other hand, when the parameters are sufficiently large, $\mathcal{V}$ is expected to contain an infinite connected component, almost surely.
This is made precise in the following theorem where we also set the existence of exponential decay for the connectivity in some regimes.
There and in the remainder of the text $p_c(\mathbb{Z}^d)$ will always stand for the critical probability for Bernoulli site percolation on $\mathbb{Z}^d$.

\begin{theorem}
\label{theo:phas_tran}
Assume that $p_i < p_c(\mathbb{Z}^{d-1})$ for some $i~ \in \{1, \ldots, d\}$, and that $p_j \neq 1$ for  some
$j ~ \in ~ \{1,\ldots,d\}\backslash \{i\}$ then
\begin{equation} 
\label{eq:phas_tran_1}
\mathbb{P}_{\mathbf{p}} (\{ \mathbf{0} \leftrightarrow \infty \}) =0.
\end{equation}
On the other hand if $p_1, \dots, p_d$ are sufficiently close to 1, then 
\begin{equation}
\label{eq:phas_tran_2}
\mathbb{P}_\mathbf{p}(\{\mathbf{0} \leftrightarrow \infty \})>0.
\end{equation}
Furthermore, if   $p_i <  p_c(\mathbb{Z}^{d-1})$ and $p_j < p_c (\mathbb{Z}^{d-1})$  for some $i \neq j  \in \{1,\ldots, d\}$, then there exists a constant $\psi = \psi(\mathbf{p},d)>0$ such that
\begin{equation}
\label{eq:expo_deca}
\mathbb{P}_{\mathbf{p}}(\{\mathbf{0} \leftrightarrow \partial B(n) \}) \leq e^{-\psi(\mathbf{p},d)n},
\end{equation}
for $n \geq 0$.

\end{theorem}

\medskip

%

Let
$
\{\mathbf{0} \leftrightarrow \partial{B}(n),~ \mathbf{0} \nleftrightarrow \infty\} 
$ 
denote the event that the origin is connected to the boundary of $B(n)$ but does not belong to a infinite cluster. 
The next theorem is the main result of this paper.
It states that there are regimes for which the connectivity decays as a power law.

\begin{theorem}
\label{theo:conn_deca_2}
Assume that $p_1 \in (0,1)$, $p_2 > p_c(\mathbb{Z}^{2})$ and $p_3 > p_c(\mathbb{Z}^{2})$.
For $d=3$, there exist constants $\alpha(\mathbf{p})>0$ and $\alpha'(\mathbf{p})>0$ depending on $\mathbf{p}$, such that
\begin{equation}
\label{eq:poly_deca}
\mathbb{P}_{\mathbf{p}} (\{ \mathbf{0} \leftrightarrow \partial{B}(n), ~ \mathbf{0} 
\nleftrightarrow \infty\}) \geq \alpha'(\mathbf{p}) n^{-\alpha(\mathbf{p})}, \text{ for all $n \geq 0$}.
\end{equation}
If $d \geq 4$, then there exists $p_{\bullet} = p_{\bullet} (p_2, p_3) \in (0,1)$ such that if $p_4 \cdots p_d \geq p_{\bullet}$ then there exist constants $\alpha(\mathbf{p})>0$ and $\alpha'(\mathbf{p})>0$ such that \eqref{eq:poly_deca} holds.
\end{theorem}

Clearly the roles of the labels $i$ in the parameters $p_i$'s can be exchanged in the statement above.
In view of Theorem \ref{theo:phas_tran} the region of $[0,1]^{d}$ covered by the parameters in the statement of Theorem \ref{theo:conn_deca_2} contains parts of the supercritical and of the subcritical regime contrasting with the well known results for Bernoulli site percolation (see the comments in the end of this section).
Note that, in the special case $d=3$ it contains entirely the supercritical regime.
It is an interesting question weather this is the case for higher dimensions.

Some ideas appearing in the proof of Theorem  \ref{theo:conn_deca_2} can be used to prove the following theorem.
\begin{theorem}
 \label{theo:pha_tra_2}
Assume that $p_2> p_c (\mathbb{Z}^2)$ and $p_3 > p_c (\mathbb{Z}^2)$, then there exists $\epsilon = \epsilon(p_2, p_3)> 0$ such that if $p_i > 1-\epsilon$, for $i=1$ and $3 < i \leq d$  then $\mathbb{P}_{\mathbf{p}}(\{\mathbf{0} \leftrightarrow \infty\}) > 0$. 
\end{theorem}
This result is not only equally good to prove the second statement in Theorem \ref{theo:phas_tran} but also adds more information about the phase diagram of the model.

We now turn our attention to the random variable $N$ defined as being the number of infinite connected components in $\mathcal{V}$.
By standard ergodicity arguments $N$ is constant almost surely.
More than that we show that it can only assume the values $0$, $1$ or $\infty$ almost surely:
\begin{theorem}
 \label{theo:num_clu}
Almost surely under $\mathbb{P}_{\mathbf{p}}$, $N$ is a constant random variable taking values in the set $\{0, 1, \infty\}$.
\end{theorem}

We finish this section by presenting a general discussion about some differences between this model and the Bernoulli site percolation, pointing out some ideas involved in the proofs of the theorems above, and finally commenting on some related percolation models.

\medskip

\noindent \textbf{Infinite-range dependencies.}

\smallskip

\noindent The model features infinite-range dependencies.
In fact, for some configurations, it is possible to conclude that an entire line passing through the origin has been removed only by examining the states of few sites lying next to the origin.
For example, in $d=3$ on the event that the origin does not belong to $\mathcal{V}$ but the sites $(1,0,0)$ and $(0,1,0)$ do, one concludes that almost surely all sites lying in the line passing through the origin and the site $(0,0,1)$ are removed.
The presence of such strong dependencies limits the application of some techniques commonly used for Bernoulli site percolation, for instance, the Peierls' arguments.

One particular difficulty is that $\mathbb{P}_{\mathbf{p}}$ is not dominated stochastically by any Bernoulli site percolation process.
Indeed, fixing a parameter $p$ and $d$ other parameters $p_1,\ldots,p_d$, all of them strictly positive, we have that for every $n$ sufficiently large,
\begin{equation*}
 \mathbb{P}_{\mathbf{p}} (\{ B(n) \subset \mathcal{V}\}) \geq \min\{p_1,\ldots,p_2\}^{d(2n+1)^{d-1}} > p^{(2n+1)^d} = \mathbb{P}_p ( \{B(n) \subset \mathcal{V} \}).
\end{equation*}
Here the first inequality follows from the fact that the number of lines touching the set $B(n)$ is, at most, equal to $d(2n+1)^{d-1}$.
Also $\mathbb{P}_{\mathbf{p}}$  does not dominates any Bernoulli site percolation since, for every $n$ large enough:
\begin{equation*}
 \mathbb{P}_{\mathbf{p}} (\{B(n) \cap \mathcal{V} \neq \emptyset\}) \leq 1 - (1-p_1)^{(2n+1)^{d-1}} < 1 - (1-p)^{(2n+1)^d} = \mathbb{P}_p (\{B(n)\cap \mathcal{V} \neq \emptyset \}).
\end{equation*}

\medskip

\noindent \textbf{Phase transition.} 

\smallskip

\noindent Theorem \ref{theo:phas_tran} settles the existence of a phase transition for the connectivity of $\mathcal{V}$ and its proof is presented in Section \ref{sec:phas_tran}.

The existence of the subcritical phase is given by \eqref{eq:phas_tran_1}.
In order to prove it we argue that, whenever $p_1 < p_c(\mathbb{Z}^{d-1})$, the connected component of $\mathcal{V}$ containing the origin is necessarily contained in a region of $\mathbb{Z}^d$ that projects into a finite subset of $\mathcal{P}_1$.
The proof is then finished by showing that, whenever $p_2 \neq 1$, any such region is always disconnected into infinitely many finites subsets as one removes lines parallel to $e_2$.

Usually the existence of a supercritical phase in a percolation model in $d\geq 3$ is proved by showing that the restriction of $\mathcal{V}$ to $\mathbb{Z}^2$ percolates provided the parameter is high enough.
However, in our case, almost surely, $\mathcal{V}\cap\mathbb{Z}^2$ does not have infinite components regardless of the choice of the parameters.
In fact, removing independently lines in $\mathbb{Z}^2$ disconnects it into infinitely many disjoint rectangles almost surely.
This difficulty can be overcome by restricting $\mathcal{V}$ to the `two-dimensional subspace' perpendicular to the vector $(1,\ldots, 1)$, instead of $\mathbb{Z}^2$.
The resulting process is a $2$-dependent percolation, that dominates a supercritical Bernoulli site percolation provided that all of the parameters $p_1, \ldots, p_d$ are sufficiently high.

\begin{remark}The fact that we only remove lines that are parallel to the coordinate axis play a very important role in this argument.
This is not the case for the cylinders percolation model \cite{Tykesson10, Hilario12} which can be regarded as a continuum and isotropic version of the model studied here (see discussion at the end of this section).
For cylinders percolation, a evolved multiscale argument needs to be used in order to show the existence of the supercritical phase when $d=3$ \cite{Hilario12}.
\end{remark}

A parameter $p_i$ is said to be subcrtical (resp.\ supercritical) if $p_i < p_c(\mathbb{Z}^d)$ (resp.\ $p_i > p_c(\mathbb{Z}^d)$). By equations \eqref{eq:phas_tran_1} it suffices to have a single one of the $d$ parameters subcritical for the model to be subcritical.
Equation \eqref{eq:phas_tran_2} shows that when all the parameters are highly supercritical, then the model is also supercritical.
It is natural to search for a complete picture of the behavior of the model when the parameters vary in $[0,1]^d$, however we do not know what happens, for instance, when the $d$ parameters are slightly supercritical.
For $d=3$, we conjecture that the model is still subcritical when all the three parameters approach the value $p_c(\mathbb{Z}^2)$ from above.
In the opposite direction, Theorem \ref{theo:pha_tra_2} shows that $\mathcal{V}$ percolates when two of the parameters are slightly supercritical, provided that all the other $d-2$ are taken sufficiently close to $1$.
Its proof is presented in the end of Section \ref{sec:sub_exp}.

\begin{remark}
For $d=3$, one can show that $\mathcal{V}$ percolates as soon as all the three parameters are strictly bigger than $(\tilde{p}_c(\mathbb{Z}^2))^{1/3}$, where $\tilde{p}_c(\mathbb{Z}^2)$ denotes  the critical percolation probability for the independent oriented site percolation model on $\mathbb{Z}^3$.
We do not present a proof of this result here see, however, the comments in Remark \ref{r:p_*}.
\end{remark}

\medskip
\noindent \textbf{Connectivity decay.}

\smallskip
 
\noindent For Bernoulli site percolation, whenever $p<p_c(\mathbb{Z}^d)$ the quantity $\mathbb{P}_p(\{\mathbf{0} \leftrightarrow \partial{B}(n)\})$  decays  exponentially fast as $n$ increases (see \cite{Menshikov86},  \cite{Aizenman87} and \cite{Duminil-Copin15}).
As a consequence, the expected volume of the connected component containing the origin is finite.
This result is sometimes referred to as `coincidence of the critical points' or `uniqueness of the phase transition'.
For our model, the result concerning equation \eqref{eq:expo_deca}  follows from the exponential decay for Bernoulli site percolation as we will show in the ending of Section \ref{sec:phas_tran}.

For Bernoulli site percolation with $p>p_c(\mathbb{Z}^d)$, the ideas in \cite{Chayes87} can be used in order to show that $\mathbb{P}_p ( \{ \mathbf{0} \leftrightarrow \partial{B}(n),~ \mathbf{0} \nleftrightarrow \infty \})$ also decays exponentially fast as $n$ increases (see \cite{Grimmett99}).
Theorem \ref{theo:conn_deca_2} shows that this property no longer holds for our model since the decay is at most polynomial for a given choice of parameters that includes the highly supercritical regime and parts of the subcritical regime.
Section \ref{sec:sub_exp} is dedicated to its proof that relies on a one-step renormalisation argument and exploits the duality of the process when restricted to the coordinate planes.

\medskip

\noindent \textbf{The $3$-dimensional case.}

\smallskip

\noindent As stated above, for $d=3$ the decay is polynomial in all the supercritical regime.
Perhaps even more interesting is the fact that there is a transition in the connectivity decay inside the subcritical phase as can be inferred by equations \eqref{eq:expo_deca} and \eqref{eq:poly_deca}.
For instance, when $p_1$ and $p_2$ are strictly bigger than $p_c(\mathbb{Z}^{2})$ and $p_3$ is smaller than $p_c(\mathbb{Z}^{2})$ then, in view of equations \eqref{eq:phas_tran_1} the quantity $\mathbb{P}_{\mathbf{p}}(\{\mathbf{0} \leftrightarrow \partial{B}(n)\})$ does not decay exponentially fast, although the model is still subcritical.

One can consider a single parameter model by choosing $p_1 = p_2 = p_3 = \rho$ and then define its critical point by 
\begin{equation}
\label{e:p_*}
p_* := \inf \{\rho \in [0,1]; \mathbb{P}_{(\rho, \rho, \rho)} (\{\mathbf{0} \leftrightarrow \infty\}) > 0 \}.
\end{equation}
As long as $\rho < p_c(\mathbb{Z}^2)$ the connectivity decay is exponential, whereas it is at most polinomial when $\rho > p_c(\mathbb{Z}^2)$.
As stated above, we conjecture that $p_* > p_c(\mathbb{Z}^2)$, so that the vacant set does not percolate when $\rho = p_c(\mathbb{Z}^2)+\delta$ for $\delta$ taken sufficiently small.
If this turns out to be true, then the model presents a transition from exponential decay to a power law decay within the subcritical regime $\rho \in [0, p_*)$.

\medskip
\begin{center}
\begin{tikzpicture}[scale=.8]
\draw (0, 0) -- (10, 0);
\draw (0, .2) -- (0, -.2);
\draw (5, .2) -- (5, -.2);
\draw (7.0, .2) -- (7.0, -.2);
\draw (10, .2) -- (10, -.2);
\draw[dotted] (7, -.7) -- (7, -1.3);
\draw[dotted] (5, .2) -- (5, .8);

\node[below] at (0, -.2) {$0$};
\node[above] at (2.5, 0.2) {exponential decay};
\node[below] at (3.5,-.6) {no infinite components};
\node[below] at (5, -.1) {$p_c(\mathbb{Z}^2)$};
\node[align=center, above] at (7.5, .2) {polynomial decay};
\node[below] at (7.0, -.2) {$p_*$};
\node[align=center, below] at (8.5, -.2) {infinite \\ components};
\node[below] at (10, -.2) {$1$};

\end{tikzpicture}
\end{center}

\medskip

\noindent \textbf{Uniqueness of the infinite cluster.}

\smallskip

\noindent For proving Theorem \ref{theo:num_clu} we use a procedure similar to that of Newman and Schulman in \cite{Newman81_1} and \cite{Newman81_2}. 
However their methods do not apply directly to the measure $\mathbb{P}_{\mathbf{p}}$ due to fact that it fails to satisfy the so-called \emph{finite energy condition}.
Thus a non-trivial extension is needed. 
For that, we use the fact that for a translation-invariant measure on $\{0,1\}^{\mathbb{Z}^d}$ all infinite connected components have a well defined density (see \cite{Burton89}) and that the percolation processes $\omega_i$ satisfy the finite energy condition. 
The question whether $N \in \{0,1\}$ is still open. 
We conjecture that it should be the case at least when the components of $\mathbf{p}$ are high.

\medskip
\noindent \textbf{Related models and motivation.}

\smallskip

\noindent The model we present is an example of so called coordinate percolation since the state of each vertex is determined by random variables 
associated  with the projections of each of the vertex into lower-dimensional subspaces of the lattice what creates the infinite-range dependencies.

There are some other examples of coordinate percolation models that have been shown to present polynomial connectivity decay.
One of these models is the so-called \textit{corner percolation} studied by G.\ Pete in \cite{Pete08}.
It consists of a bond percolation model in $\mathbb{Z}^2$ in which every edge is removed with probability $1/2$, conditional that every site has exactly two perpendicular incident edges.
In \cite{Pete08} it is shown that every connected component of the vacant set is a closed circuit, so there are no infinite connected components.
It is shown that the probability that the diameter of the circuit containing the origin is greater than $n$ decays polynomially fast in $n$.

Another coordinate percolation model defined in $\mathbb{Z}^2$ is the so called \emph{Winkler's percolation}. This is a percolation representation due to N.~Alon for a model of scheduling of non-colliding random walks in the complete graph introduced by P.\ Winkler \cite{Winkler00} with the goal to study a problem in distributed computation.
In this percolation representation, for a fixed $k \geq 4$, one assigns independently to each horizontal and vertical line of the lattice an element chosen uniformly at random in $\{1,\ldots, k\}$.
A given site is said to be present (or unblocked) if the values assigned to the two lines passing through it differ.
Otherwise it is said to be removed (or blocked).
The main question is whether the probability that there exists an infinite oriented path in the vacant set starting at the origin is positive.
Allowing non-oriented paths, it has been shown in \cite{Winkler00} and \cite{Balister00} that for $k \geq 4$ the answer to this question is affirmative.
In \cite{Gacs11} and more recently, in \cite{Basu14} it was shown that this is also true for the original oriented setting provided that $k$ is very large.
The results of \cite{Gacs00} show that equation \eqref{eq:poly_deca} holds for this model in the supercritical regime ($k \gg 1$).

Despite from the similarity with Winkler's percolation, the percolation process considered here was inspired by a rather different dependent site percolation model, the \emph{random interlacements} introduced in \cite{Sznitman07} by Sznitman.
There the object of interest is the complementary set of sites of a Poisson point process in the space of double-infinite trajectories on $\mathbb{Z}^d$ modulo time shifts.
In \cite{Sznitman07} and in \cite{Sidoravicius09} it is shown that this model undergoes a phase transition on a non-degenerate multiplicative parameter for the intensity measure $u_*$.
In \cite{Sidoravicius10} it is proved that the connectivity function has a stretched exponential decay for intensities larger then a second parameter value $u_{**} \geq u_{*}$, which is related to the disconnection time of cylinders by random walks trajectories.
It is not known whether $u_{**}$ and $u_*$ coincide.

There are some other models where infinite straight lines are drilled from the ambient space.
Tykeson and Windisch \cite{Tykesson10} have studied the phase transition for a continuum analogous percolation model called \emph{cylinders' percolation} proposed by I.~Benjamini.
The vacant set for this model consist of the set of points of $\mathbb{R}^d$ that is not covered by the union of a collection of bi-infinite cylinders of radius one and with axis given by the lines in a realization of a Poisson process on the space of lines in $\mathbb{R}^d$.
They prove that for  $d \geq 3$ when the intensity of underlying Poisson process is high the vacant set does not percolate.
For dimension $d \geq 4$ they also showed that the vacant set percolates when the intensity is sufficiently small.
More than that, they showed that for any $d \geq 4$, if the intensities are low enough then, the restriction of $\mathcal{V}$ to any fixed subspaces of dimension $2$ has an infinite connected component almost surely.
For $d =3$ this is not the case: They showed that the probability of having percolation on any subspace of dimension $2$ is zero.
In \cite{Hilario12} it is shown that for $d=3$ the vacant set percolates in a thick enough slab for sufficiently small intensities, completing the proof of the existence of a non-trivial phase transition for $d=3$.

\medskip

\noindent \textbf{Simulations and the critical behavior.}

\smallskip

\noindent After the final draft of this paper  was ready, we discovered that the exactly same model was already introduced in the physics literature by  Y. Kantor \cite{Kantor86} for $d=3$.
Using Monte Carlo simulations the author obtained the value of $p_*=0.6345 \pm 0.0003$ and also an estimate on the statical critical exponent for the correlation length ($\nu$).
Remarkably, he also obtained estimates for the  the fractal co-dimension of the infinite cluster which are very close to the known value for the $3$-dimensional Bernoulli percolation what led him to  point out that the two models could belong to the same universality class.
The dynamical critical exponent for the shortest path spanning a big box of the lattice from bottom-to-top was also estimated and shown to coincide within error bars with the one for $3$-dimensional Bernoulli site percolation.
In \cite{Schrenk15}, more Monte Carlo simulations were performed.
The critical point  $p_*$ and the exponents obtained are within error bars consistent to the ones obtained in \cite{Kantor86}.
Also several other statical and dynamical critical exponents were obtained.
Although some of them, including $\nu$ are shown to differ from the the $3$-dimensional Bernoulli site percolation, the data suggest that they satisfy the usual scaling relations.

\medskip

\noindent \textbf{Open Problems.}

\smallskip

\noindent We finish this section presenting a list of some interesting questions left open.
\medskip

\noindent \textbf{Problem 1:} Show that for $d=3$, $p_* > p_c(\mathbb{Z}^2)$, where $p_*$ is defined in \eqref{e:p_*}.
Does the same hold for $d\geq 4$?
\medskip

\noindent \textbf{Problem 2:} For $d \geq 4$, does power-law decay of connectivity hold for all the supercritical phase?
\medskip

\noindent \textbf{Problem 3:} Is the infinite connected component unique when it exists?

\section{Mathematical setting}
\label{sec:math_sett}

In this section we present a construction of the Bernoulli line percolation model described in the previous 
section and fix some of the notation to be used throughout the text.

As in the previous section, let $\{e_1, \ldots, e_d\}$ represent the canonical orthonormal basis of
 $\mathbb{Z}^d$ and $\mathcal{P}_i$ be the vector subspace of $\mathbb{Z}^d$ orthogonal to $e_i$.
Note that there exists a natural isomorphism between $\mathcal{P}_i$ and $\mathbb{Z}^{d-1}$.
We denote by $\pi_i : \mathbb{Z}^d \to \mathcal{P}_i$ the orthogonal projection from $\mathbb{Z}^d$ to $\mathcal{P}_i$.

Consider for each $i \in \{1, \dots, d \}, $ the space $\Omega_i = \{0,1\}^{\mathcal{P}_i}$ endowed with $\mathcal{F}_i$, the canonical $\sigma$-algebra generated by the cylinder sets.
Elements of $\Omega_i$ will be called configurations in $\mathcal{P}_i$ and will be denoted by $\omega_i$.
For a fixed parameter $p_i$, let $\mathbb{P}_{p_i}$ be the canonical product probability measure on 
$(\Omega_i, \mathcal{F}_i)$ under which $\{\omega_i(v);~v \in \mathcal{P}_i\}$ is a family of independent 
Bernoulli random variables with mean $p_i$.

Now take $\Omega = \{0,1\}^{\mathbb{Z}^d}$ endowed with the product $\sigma$-algebra $\mathcal{F}$ 
and define for each site $v \in \mathbb{Z}^d$,
\begin{equation}
\label{eq:omeg}
\omega (v) = \prod_{i=1}^d \omega_i (\pi_i (v)).
\end{equation}
Thus, $\omega = \{ \omega(v);~ v \in \mathbb{Z}^d\}$ is a $\Omega$-valued random 
element defined on  the product space $(\Omega_1 \times \cdots \times \Omega_d, \mathcal{F}_1\times \cdots \times \mathcal{F}_d)$.
Elements in $\Omega$ will be called configurations in $\mathbb{Z}^d$.
 Let $\mathbf{p} = (p_1,\ldots, p_d) \in [0,1]^d$. Denote by $\mathbb{P}_\mathbf{p}$ the distribution of 
 $\omega$ in $\Omega$, \textit{i{.}e{.}}, for each $\mathcal{A} \in \mathcal{F}$
\begin{equation}
\label{eq:coor_perc}
\mathbb{P}_{\mathbf{p}} (\mathcal{A}) = \mathbb{P}_{p_1} \times \cdots \times \mathbb{P}_{p_d} (\{ \omega \in \mathcal{A}\}).
\end{equation}
\begin{definition}
The law $\mathbb{P}_{\mathbf{p}}$ will be called {Bernoulli line percolation} on $\mathbb{Z}^d$ with 
parameter vector $\mathbf{p} = (p_1, \ldots, p_d)$.
\end{definition}

A component $p_i$ of the vector $\mathbf{p}= (p_1, \ldots, p_d)$ is said to be subcritical 
(supercritical) if it is strictly smaller (strictly larger) than 
$p_c(\mathbb{Z}^{d-1})$.
 
For a fixed $i \in \{1,\ldots, d\}$ and a given $\omega_i \in \Omega_i$, each $v_i \in \mathcal{P}_i$ is 
said to be $\omega_i$-open if $\omega_i(v_i) = 1$ and $\omega_i$-closed otherwise.
Similarly, for a given $\eta \in \Omega$, $v$ is said to be $\eta$-open (or simply open) if $\eta(v) = 0$ 
and $\eta$-closed otherwise.
Then it follows that 
\[
\text{a site $v \in \mathbb{Z}^d$ is $\omega$-open if and only if $\pi_i(v)$ is $\omega_i$-open for all $i=1,\ldots,d$}.
\]
The connection with the formulation presented in Section \ref{sec:intr} is immediately made by identifying 
the set of removed sites with the set of $\omega$-closed sites.

For a fixed $\eta \in \Omega$, let us define the vacant set
\begin{equation}
 \label{eq:vaca}
\mathcal{V}(\eta) = \{ v \in \mathbb{Z}^d;~ \eta(v) =  1\}.
\end{equation}
Our goal is to study the connectivity properties of the random set $\mathcal{V}$ under the 
Bernoulli line percolation measure $\mathbb{P}_{\mathbf{p}}$ as we vary the parameter vector 
$\mathbf{p} = (p_1, \ldots, p_d)$.

A nearest-neighbors path (or simply a path) in $\mathbb{Z}^d$ is a finite or infinite chain of sites 
$v_1$, $v_2$,$\ldots$ all distinct from each other and such that $v_i$ is a neighbor of $v_{i+1}$ 
in the $\mathbb{Z}^d$-lattice. 
The origin in $\mathbb{Z}^d$ will be denoted by $\mathbf{0}$. 
For simplicity we will also write 
$\mathbf{0}$ to refer to $\pi_i(\mathbf{0})$ the origin of $\mathcal{P}_i$.
We say that $\mathcal{V}$ percolates or simply that there is percolation if 
$\mathbb{P}_{\mathbf{p}} (\{\mathbf{0} \leftrightarrow \infty \}) > 0$, where as in the last 
section $\{\mathbf{0} \leftrightarrow \infty\}$ stands for the event that there is an infinite self-avoiding path of nearest-neighbors sites contained in $\mathcal{V}$ and starting at the origin $\mathbf{0}$ of $\mathbb{Z}^d$.
Note that $\mathbb{P}_{\mathbf{p}}$ is invariant under the action of the group of translations that preserve the $\mathbb{Z}^d$-lattice.
Moreover it is ergodic (and even mixing) under such a translation provided that it does not preserve any of the coordinate axis.
Then $\mathbb{P}_{\mathbf{p}} (\{ \mathbf{0} \leftrightarrow \infty \} ) > 0$ is equivalent to $\mathbb{P}_{\mathbf{p}}(\{\mathcal{V} \text{ contains an infinite connected component} \}) = 1$.

For each nonnegative integer, let $B^d(n) = [-n,n] \cap \mathbb{Z}^d$ be the $l_\infty$-ball of radius $n$ around the origin and for each $v \in \mathbb{Z}^d$ let $B^{d}(v, n) = B^{d}(n) + v$.
We denote by $\partial{B}^d(v,n)$ the set of sites in $B^d(v,n)$ that have at least one of its nearest neighbors outside $B^d(v,n)$.
We also let $\{\mathbf{0} \leftrightarrow \partial{B}(n)\}$ be the event that there is a nearest-neighbors path in $\mathcal{V}$ connecting $\mathbf{0}$ to $\partial{B}^d(n)$.
Recalling that $\mathcal{P}_i$ is isomorphic to $\mathbb{Z}^{d-1}$ we write $B^{d-1}_i(n)$ for the $l_{\infty}$-ball of radius $n$ centered at the origin $\pi_i(\mathbf{0})$ of $\mathcal{P}_i$.
We also write $\{\mathbf{0} \leftrightarrow \infty\}$ and $\{\mathbf{0} \leftrightarrow \partial{B}^{d-1}_i(n)\}$ for the corresponding events in $\Omega_i$.

For a given configuration $\eta \in \Omega$, the set $C(\mathbf{0}) = C(\mathbf{0},\eta)$ is the set of points in $\mathbb{Z}^d$ that can be reached from $\mathbf{0}$ by a $\eta$-open path starting at $\mathbf{0}$.
This set will be called the connected component at the origin for the configuration $\eta$, or sometimes the connected component of $\eta$ containing the origin.
Similarly we define $C_i(\mathbf{0}) = C_i(\mathbf{0}, \omega_i)$ connected component of $\omega_i$ containing the origin.

\section{Phase transition and exponential decay}
\label{sec:phas_tran}
\label{tail_decay}

In this section we prove Theorem \ref{theo:phas_tran}, that establishes the existence of a non-trivial phase transition for $\mathbb{P}_{\mathbf{p}}$ and the exponential decay for the tail distribution of the radius of the connected component containing the origin, when least two of the parameters are subcritical.

The proof of Theorem \ref{theo:phas_tran} is divided in two steps.
First we prove the first assertion showing that $\mathcal{V}$ does not percolate when at least one  of the 
components of $\mathbf{p}$ is strictly smaller then $p_c(\mathbb{Z}^{d-1})$ and  any other of its parameters is not equal to one.

In the remainder of this section, we denote for short $\mathbb{P} = \mathbb{P}_{p_1} \times \ldots \times \mathbb{P}_{p_d}$ the law of $(\omega_1, \ldots, \omega_d)$ in $\Omega_1 \times \ldots \times \Omega_d$.

Recalling that
\[
\omega(v) = \prod_{i=1,\ldots, d} \omega_i (\pi_i(v)),
\]
we conclude that the $\pi_i$-projection of a nearest-neighbor path in $\mathbb{Z}^d$ containing two vertices $v$ and $w$ is a nearest-neighbor path (possibly consisting of a single site) in $\mathcal{P}_i$ connecting $\pi_i(v)$ to $\pi_i(w)$. 
Thus $\pi_i(C(v))$ is a connected subset of $C_i(\pi_i(v))$.
Let us denote
\[
 \mathcal{A}_i(n) = \{ \eta \in \Omega;~ \pi_i(C(\mathbf{0})) \text{ contains a path connecting } \mathbf{0} \text{ to } {B}_i ^{d-1}(n) \}.
\]
\begin{lemma}
\label{lemma:tube}
Assume that $p_2< 1$.
Then for any nonnegative integer $n_0$, 
\begin{equation}
\label{eq:lim_P}
\lim_{n\to\infty} \mathbb{P} (\{\omega \in \mathcal{A}_2(n)\} \cap \{C_1(\mathbf{0}) \subset {B}_1^{d-1}(n_0)\})=0.
\end{equation}
\end{lemma}

\begin{proof}
Since $\pi_1(C(\mathbf{0})) \subset  C_1(\mathbf{0})$, if  $C_1(\mathbf{0}) \subset {B}_1^{d-1}(n_0)$ then $\pi_1(C(\mathbf{0})) \subset {B}_1^{d-1}(n_0)$.
Thus
\[
 C(\mathbf{0}) \subset \pi_1^{-1}({B}_{1}^{d-1}(n_0)) = \mathbb{Z} \times [-n_0, n_0] \times \cdots \times [-n_0, n_0].
\]
Therefore, denoting 
\[
S_2(n_0) := \mathbb{Z} \times \{0\}  \times [-n_0,n_0] \times \cdots \times [-n_0,n_0] \subset \mathcal{P}_2
\]
(see Figure \ref{fig:tube}) it follows that if $C_1(\mathbf{0}) \subset {B}_1^{d-1}(n_0)$, then  $\pi_2(C(\mathbf{0})) \subset S_2 (n_0)$, and so
\[
 \{ \omega \in \mathcal{A}_2(n)\}\cap\{C_1(\mathbf{0}) \subset {B}_1^{d-1}(n_0)\} \subset \{\mathbf{0} \leftrightarrow \partial{B}_2^{d-1}(n) \text{ in } S_2(n_0) \}.
\]
where the last event in the equation above is a cylinder set in $\Omega_2$.
So we have that
\begin{equation}
\label{eq:lim_P_2}
\mathbb{P}(\{\omega \in \mathcal{A}_2(n)\}\cap\{C_1(\mathbf{0}) \subset {B}_1^{d-1}(n_0)\}) \leq \mathbb{P}_{p_2}(\{\mathbf{0} \leftrightarrow \partial{B}_2^{d-1}(n) \text{ in } S_2 (n_0) \}).
\end{equation}
For a $n>n_0$, if $\{ \mathbf{0} \leftrightarrow \partial{B}_2^{d-1}(n) \text{ in } S_2(n_0) \}$ occurs, then by the definition of $S_2$, we must have that either
\[
\begin{array}{l}
\mathcal{B}_2(n) := \{ \mathbf{0} \leftrightarrow (\{n\} \times \{0\} \times [-n_0,n_0] \times \cdots \times [-n_0, n_0]) \text{ in } S_2(n_0) \} \text{ or }\\
\mathcal{B}_2(-n) := \{ \mathbf{0} \leftrightarrow (\{-n\} \times \{0\} \times [-n_0,n_0] \times \cdots \times [-n_0, n_0]) \text{ in } S_2(n_0) \}
\end{array}
\]
occurs.
Since those two events have the same probability, it follows that 
\begin{equation}
\label{eq:lim_P_3}
\limsup_{n\to\infty}\mathbb{P}_{p_2}(\{\mathbf{0} \leftrightarrow \partial{B}_2^{d-1}(n) \in S_2(n_0) \}) \leq 2 \limsup_{n\to\infty} \mathbb{P}_{p_2} (\mathcal{B}_2(n)).
\end{equation}

Consider now, for each $k \in \mathbb{N}$ the events,
\[
 \mathcal{C}_2(k) = \{ \omega_2 \in \Omega_2;~\omega_2(v) = 0 \text{ for all } v \in \{k\} \times \{0\} \times [-n_0,n_0] \times \cdots \times [-n_0, n_0] \}.
\]
The events $\mathcal{C}(k)$, for $k \in \mathbb{Z}$ are independent with $\mathbb{P}_{p_2} (\mathcal{C}_2(k)) = (1-p_2)^{(2n_0+1)^{d-2}}>0$ since, by hypothesis, $p_2 < 1$.
Then
\[
\mathbb{P}_{p_2}(\{\mathcal{C}_2(k) \text{ occurs infinitely often}\}) = 1.
\]

Since $\mathcal{B}_2(n) \subset \cap_{k=1}^{n} \mathcal{C}_2(k)^c$, we have that $\lim_{n\to\infty}\mathbb{P}_{p_2}(\mathcal{B}_2(n)) = 0$.
The result follows then from \eqref{eq:lim_P_2} and \eqref{eq:lim_P_3}.
\end{proof}

\begin{figure}
\label{fig:tube}
\begin{tikzpicture}[scale=.7]
\draw (-6,-3) -- (0,0) -- (8,0) -- (2,-3) -- cycle;
\coordinate (o) at (.5,-1.5);
\draw[fill=gray!20] plot [smooth cycle] coordinates {($(o)+(-1.5,-.25)$) ($(o)+(-.5,.25)$) ($(o)+(1,.45)$) ($(o)+(0,-.45)$)};
\draw[fill=gray!20] plot [smooth cycle] coordinates {($(o)+(-1.5,-.25)+(0,6)$) ($(o)+(-.5,.25)+(0,6)$) ($(o)+(1,.45)+(0,6)$) ($(o)+(0,-.45)+(0,6)$)};
\node[above left] at ($(o)+(1,.5)$) {\tiny$C_1(\mathbf{0})$};
\node[below left] at ($(o)+((0.1,.1)$) {\tiny$\mathbf{0}$};
\fill (o) circle [radius=1pt];
\draw[dashed] ($(o)+(1,.5)$) -- (-2,-1) -- (-2,6) ($(o)+(0,-.5)$) -- (-4,-2) -- (-4,5);
\draw[dashed, red] ($(o)+(1,.5)+(0,4)$) -- (-2,3) ($(o)+(0,-.5)+(0,4)$) -- (-4,2);
\draw[dashed] ($(o)+(1,.5)+(0,6)$) -- (-2,5) ($(o)+(0,-.5)+(0,6)$) -- (-4,4) -- (-2,5);
\draw (0.25,-1) -- (-1.75,-2) -- (.75,-2) -- (2.75, -1) -- cycle;
\node[right] at (2.75, -1)  {\tiny$B^{d-1}_{1}(n_0)$};
\draw[red] (-2,3) -- (-4,2);
\node[below left] at (-4,2) {\tiny$k$};
\node[below left] at (-4,4) {\tiny$n$};
\draw[red, fill=red!20] plot [smooth cycle] coordinates {($(o)+(-1.5,-.25)+(0,4)$) ($(o)+(-.5,.25)+(0,4)$) ($(o)+(1,.45)+(0,4)$) ($(o)+(0,-.45)+(0,4)$)};
\draw[dotted] ($(o)+(-1.5,-.25)$) -- ($(o)+(-1.5,-.25)+(0,6)$) ($(o)+(1,.5)$) -- ($(o)+(1,.5)+(0,6)$) ($(o)+(0,-.5)$) -- ($(o)+(0,-.5)+(0,6)$);
\fill (-3,-1.5) circle [radius=1pt] (-2,-1) circle [radius=1pt] (-4,-2) circle [radius=1pt];
\node[above] at (-3, -1.5) {\tiny$0$};
\node[below] at (-2, -1) {\tiny$n_0$};
\node[below] at (-4, -2) {\tiny$-n_0$};
\node[below] at (-3,6.5) {\tiny$S_2(n_0)$};
\end{tikzpicture}
\caption{The cluster containing the origin in $\mathcal{P}_1$ is contained in the box $B^{d-1}_1$.
This implies that, the cluster containing $\mathbf{0}$ is restricted to the vertical 'tubular' region contained in the interior of $\pi_1^{-1} (B^{d-1})$. The red horizontal crossing inside $S_2(n_0)$ represents the closed sites in the definition of the event $\mathcal{C}_2(k)$. Note that the existence of infinitely many such crossings disconnects the 'tubular' region.}
\end{figure}
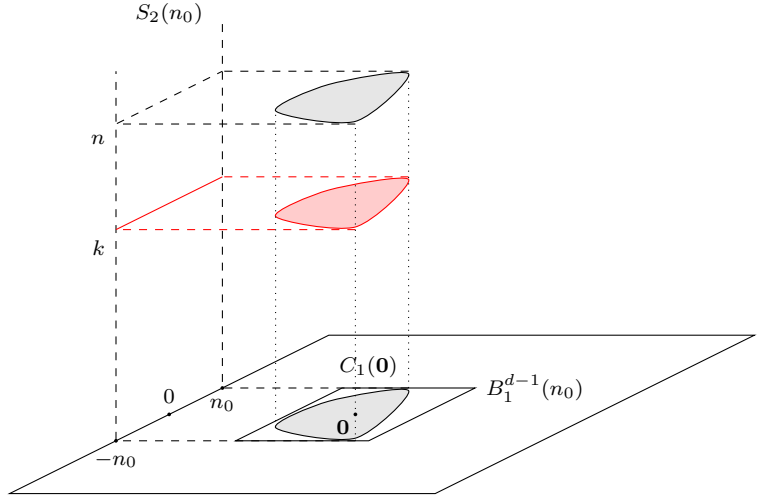

\begin{lemma}
\label{lemma:tube_2}
Assume that $p_2 <1$. Then
\begin{equation}
\label{eq:tube}
\mathbb{P} ( \{\omega \in \{\mathbf{0} \leftrightarrow \infty\}\} \cap \{C_1(\mathbf{0}) \subset {B}_1^{d-1}(n_0)\} ) = 0.
\end{equation}
\end{lemma}
\begin{proof}
Whenever $\omega \in \{C_1(\mathbf{0}) \subset {B}_1^{d-1}(n_0)\}$, then $\omega \notin \mathcal{A}_1(n)$ for all $n> n_0$.
On the other hand, whenever $\omega \in \{\mathbf{0} \leftrightarrow \partial{B}^d(n)\}$ then, for at least $d-1$ indices $i \in \{1,\ldots,d\}$ we have that $\omega_i \in \mathcal{A}_i(n)$.
Indeed, the ending point of the a path from $\mathbf{0}$ to $\partial B^{d-1}(n)$ is a vertex containing at least one coordinate equal to $n$.
Assuming, without loss of generality, that this is the last coordinate, then the $\pi_1, \ldots, \pi_{d-1}$ projections of this vertex also have the last coordinate equal to $n$.

Noting that $\{\mathbf{0} \leftrightarrow \infty\}= \displaystyle{\bigcap_{n\geq 1}} \{\mathbf{0} \leftrightarrow \partial{B}^d(n)\}$ it follows that,
\begin{equation}
 \{\omega \in \{\mathbf{0} \leftrightarrow \infty\}\} \cap \{C_1(\mathbf{0}) \subset {B}_1^{d-1}(n_0)\} \subset \bigcap_{n>n_0} \{\omega \in \mathcal{A}_2 (n)\}\cap\{C_1(\mathbf{0}) \subset {B}_1^{d-1}(n_0)\}.
\end{equation}
The result follows immediately by computing the probability of both sides of the above equation and then using equation \eqref{eq:lim_P} in the right-hand side.
\end{proof}

We now present the proof that the model undergoes a non-trivial phase transition.

\begin{proof}[Proof of Theorem \ref{theo:phas_tran}:]
We start by proving the first assertion in Theorem \ref{theo:phas_tran} which is a direct consequence of equation \eqref{eq:tube}.

Let us assume, without loss of generality that $p_1 < p_c(\mathbb{Z}^{d-1})$ and $p_2<1$.
Then
 \[ 
 \mathbb{P}_{p_1} \Big( \bigcup_{n_0=1}^{\infty} \{C_1(\mathbf{0}) \subset B_1^{d-1}(n_0) \}\Big) = \mathbb{P}_{p_1}( \{\mathbf{0} \nleftrightarrow \infty\}) =1,
 \]
thus using  equation \eqref{eq:tube} in Lemma \ref{lemma:tube_2} we obtain
\begin{equation}
\begin{array}{l}
\mathbb{P}_{\mathbf{p}}(\{\mathbf{0} \leftrightarrow \infty\}) = \mathbb{P} (\{ \omega \in \{\mathbf{0} \leftrightarrow \infty\}\}) = \\
=  \mathbb{P} \Big( \displaystyle{\bigcup_{n_0=1}^{\infty}} \{ \omega \in \{\mathbf{0} \leftrightarrow \infty\}\} \cap \{C_1(\mathbf{0}) \subset B_1^{d-1}(n_0) \}\Big) =0.
\end{array}
\end{equation}

We now prove the second assertion of Theorem \ref{theo:phas_tran}.
Fix a integer $k$ and let
\begin{equation}
\label{eq:graph}
\mathcal{P} = \{v = (v_1, \ldots, v_d) \in \mathbb{Z}^d;~v_i = 0 \text{ for } i>3, \text{ and } \sum_{i=1}^3 v_i \in  \{k-1, k, k+1\}\} ,
\end{equation}
seen as a subgraph.
Roughly speaking $\mathcal{P}$ is a connected subset of $\mathbb{Z}^d$ resembling a $2$-dimensional affine plane with normal vector $(1,1,1, 0,\ldots,0)$, (see Figure \ref{fig:G} for an illustration of $\mathcal{P}$ when $d=3$).
Noting that $\mathcal{P}$ has bounded degree and using an straightforward comparison with percolation in the honeycomb lattice, it is standard that $\mathcal{P}$ has a nontrivial critical point for Bernoulli site percolation.

\begin{figure}
\label{fig:G}

\begin{tikzpicture}[scale=.6]
\coordinate (x1) at (0,0);
\coordinate (x2) at ($(x1)+(210:1)+(0,-1)$);
\coordinate (x3) at ($(x1)+(330:1)+(0,-1)$);
\coordinate (x4) at ($(x2)+(210:1)+(0,-1)$);
\coordinate (x5) at ($(x3)+(210:1)+(0,-1)$);
\coordinate (x6) at ($(x3)+(330:1)+(0,-1)$);
\coordinate (x7) at ($(x4)+(210:1)+(0,-1)$);
\coordinate (x8) at ($(x5)+(210:1)+(0,-1)$);
\coordinate (x9) at ($(x6)+(210:1)+(0,-1)$);
\coordinate (x10) at ($(x6)+(330:1)+(0,-1)$);
\coordinate (x11) at ($(x7)+(210:1)+(0,-1)$);
\coordinate (x12) at ($(x8)+(210:1)+(0,-1)$);
\coordinate (x13) at ($(x9)+(210:1)+(0,-1)$);
\coordinate (x14) at ($(x10)+(210:1)+(0,-1)$);
\coordinate (x15) at ($(x10)+(330:1)+(0,-1)$);

\foreach \i in {1,...,15}
 \foreach \j in {0,...,5}
  \foreach \k in {1,...,3}
 {
\fill (x\i) circle [radius=2pt];
\fill ($(x\i)+(30+\j*60:1)$) circle [radius=2pt];
\draw[thick] ($(x\i)+(30+\j*60:1)$) -- ($(x\i)+(-30+\j*60:1)$) (x\i) -- ($(x\i)+(-30+\k*120:1)$);
}

\draw[->] (0,1) -- +(0,2);
\draw[->] ($(x15)+(330:1)$) -- +(330:2);
\draw[->] ($(x11)+(210:1)$) -- +(210:2);
\draw[dotted] (0,0)--(0,-4) (x15)--(0,-4) (x11)--(0,-4);

\node[right] at (0,3) {\tiny$n e_1$};
\node[below] at ($(x15)+(330:3)$) {\tiny$n e_2$};
\node[below] at ($(x11)+(210:3)$) {\tiny$n e_3$};

\node[right] at (-0.2,1.3) {\tiny$k$};
\node[left] at ($(x11)+(205:1)$) {\tiny$k$};
\node[right] at ($(x15)+(335:1)$) {\tiny$k$};

\node at (210:4.5) {\tiny$\mathcal{P}_2$};
\node at (330:4.5) {\tiny$\mathcal{P}_3$};
\node at (0,-8) {\tiny$\mathcal{P}_1$};
\end{tikzpicture}
\caption{This figure shows a portion of the graph $G$ when $d=3$ and $k$ is chosen to be equal to $5$ in \eqref{eq:graph}. Note that the projections of each vertices into $\mathcal{P_i}$ coincides at most two other vertices.}
\end{figure}
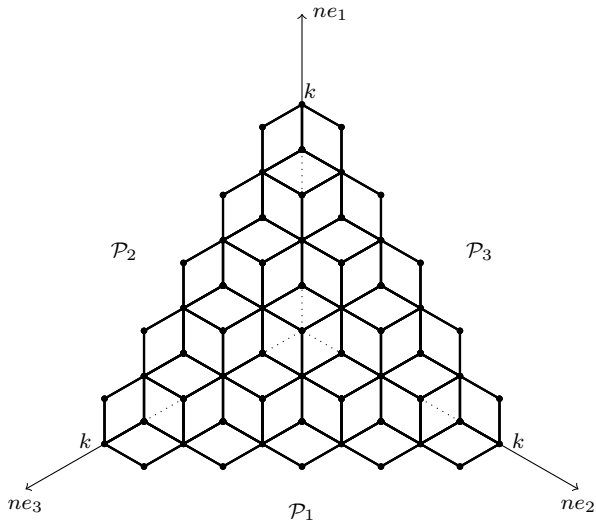

Note that, two sites in $\mathcal{P}$ can only be projected into the same point in $\mathcal{P}_i$ if they both belong to the set $\{v, v+e_i, v-e_i\}$ for some $v \in \mathcal{P}$.
So, in view of \eqref{eq:coor_perc} we conclude that $\{\omega(u);~u \in \mathcal{P}\}$ is a two-dependent percolation process in $\mathcal{P}$ in which each site $u$ is open with marginal probability $\rho = p_1\cdots p_d$.
By \cite[Theorem 0.0]{liggett97}, if we take $\rho$ large enough, then the restriction of $\mathbb{P}_{\mathbf{p}}$ to $\{0,1\}^\mathcal{P}$ dominates stochastically a supercritical Bernoulli site percolation on $\{0,1\}$.
Clearly this can be accomplished by making $p_i$ large enough for all $i=1,\ldots, d$.
Therefore, provided that all the parameters $p_i$ are large enough, we have that
$ \mathbb{P}_{\mathbf{p}}(\{ \mathbf{0} \leftrightarrow \infty \text{ in } \mathcal{P} \}) > 0$.
Since, $\mathbb{P}_{\mathbf{p}}(\{ \mathbf{0} \leftrightarrow \infty\}) \geq \mathbb{P}_{\mathbf{p}}(\{\mathbf{0} \leftrightarrow \infty \text{ in } \mathcal{P}\})$, the proof is finished.

Let us now prove (\ref{eq:expo_deca}) which is a direct consequence of the theorem of Menshikov \cite{Menshikov86} and Aizenman and Barsky \cite{Aizenman87} on the exponential decay for Bernoulli site percolation.
Assuming without loss of generality, that $p_1<p_c(\mathbb{Z}^{d-1})$ and $p_2<p_c(\mathbb{Z}^{d-1})$ are subcritical parameters of $\mathbf{p}$ then there are positive constants $\psi_1 = \psi_1(p_1)$ and $\psi_2 = \psi_2(p_2)$ such that
\begin{equation}
\label{eq:mens}
\mathbb{P}_{p_i} (\{\mathbf{0} \leftrightarrow \partial{B}^{d-1}(n)\}) \leq e^{-\psi_i(p_i) n}, \text{ for $i = 1,2$ and for all $n \geq 1$}.
\end{equation}

If $\omega \in \{ \mathbf{0} \leftrightarrow \partial B^d(n) \}$ then  $\omega_k \in \left\{ \mathbf{0} \leftrightarrow \partial{B_k^{d-1}(n)} \right\}$ for at least $d-1$ different indices $k \in \{1,\ldots, d\}$.
Then, writing $\mathbb{P}$ for $\mathbb{P}_{p_1} \times \cdots \times \mathbb{P}_{p_d}$ we have:
\begin{equation}
\label{eq:theo_expo_deca_2}
\mathbb{P}_{\mathbf{p}} \left( \{ \mathbf{0} \leftrightarrow \partial B^d(n) \} \right) \leq 
 \sum_{ \substack{A \subset \{1,\ldots,d\} \\ |A| = d-1}} \mathbb{P} \left( \cap_{j\in A} \in \left\{ \mathbf{0} \leftrightarrow \partial{B_j^{d-1}(n)} \right\} \right)
\end{equation}
Any fixed subset $A \subset \{1,\ldots,d\}$ with $|A| =d-1$ must contain at least one of the indices $1$ or $2$.  
By equation (\ref{eq:mens}) we have that:
\[ 
\mathbb{P} \left( \cap_{j \in A} \left\{ \mathbf{0} \leftrightarrow \partial{B_j^{d-1}(n)} \right\} \right) \leq \exp \left(-\psi_1(p_1)\right) + \exp \left(-\psi_2(p_2) \right) \leq \exp \left(- \alpha(\mathbf{p})n \right),
\]
for all $n \geq 1$ and for some $\alpha(\mathbf{p})>0$.
Plugging that into equation (\ref{eq:theo_expo_deca_2}) yields:
\[
\mathbb{P}_{\mathbf{p}} \left( \{ \mathbf{0} \leftrightarrow \partial B^d(n) \} \right) \leq d \exp \left(- \alpha(\mathbf{p})n \right) \text{ for all $n \geq 1$},
\]
finishing thus the proof of equation \eqref{eq:expo_deca}.
\end{proof}
\begin{remark}
As it can be seen from the proof of equation \eqref{eq:phas_tran_1}, the hypothesis that $p_i < p_c(\mathbb{Z}^{d-1})$ was only used through the fact that it implies that $\mathbb{P}_p(\{ \mathbf{0} \leftrightarrow \infty \}) =0$.
In particular, for $d=3$ and $d \gg 1$ (\textit{e.g.} $d \geq 20$) it can be replaced by the condition $p_i \leq p_c(\mathbb{Z}^{d-1})$.
\end{remark}

\section{Polynomial decay}
\label{sec:sub_exp}

This section is mainly dedicated to the proof of Theorem \ref{theo:conn_deca_2}.
\medskip

\noindent \textbf{Restricting the vacant set to $\mathbb{Z}^3$}

As it will become clear below, it will be useful to restrict the model to a $3$-dimensional vector subspace of $\mathbb{Z}^d$.
For that, let us denote by $\mathcal{S}$ the vector subspace of $\mathbb{Z}^d$ generated by the vectors $e_1$, $e_2$ and $e_3$ (later on we will identify $\mathcal{S}$ with the cubic lattice $\mathbb{Z}^3$ itself).
Given $v \in \mathcal{S}$ we write
\[
\omega'(v) : = \omega_1(\pi_1(v)) \omega_2(\pi_2(v)) \omega_3(\pi_3(v)) \,\,\, \text{ and } \,\,\, \omega''(v) := \prod_{i =4}^{d} \omega_i (\pi_i (v))
\]
with the convention that $\omega'' \equiv 1$ in case $d =3$.
It follows from \eqref{eq:omeg} that, for each $v$ in $\mathcal{S}$,
\begin{equation}
\omega(v) = \omega'(v) \omega''(v).
\end{equation}

Naturally identifying $\mathcal{S}$ with $\mathbb{Z}^3$, it makes sense to speak of $3$-dimensional Bernoulli line percolation in $\mathcal{S}$.
Indeed one can check that 
\begin{equation}
\label{eq:rand_fiel_omeg}
\begin{array}{c}
\text{the random field $\{ \omega' (v); v \in \mathcal{S} \}$ has the law of a 3-dimensional } \\ 
\text{Bernoulli line percolation process in $\mathcal{S}$ with parameters $\mathbf{p}' := (p_1, p_2, p_3)$}.
\end{array}
\end{equation}
For $v \in \mathcal{S}$, we have that $\pi_i(v) = v$ for any $i=4,\ldots,d$.
Therefore one can also check that
\begin{equation}
\label{eq:rand_fiel_omeg'}
\begin{array}{c}
\text{the random field $\{ \omega''(v); v \in \mathcal{S} \}$ has the law of a Bernoulli} \\ \text{site percolation process on $\mathcal{S}$ with parameter $p'':= p_4 \cdots p_d$},
\end{array}
\end{equation}
(with the convention $p''=1$ in case $d = 3$).

From now on we do not distinguish between $\mathcal{S}$ and $\mathbb{Z}^3$.
Then, in view of equations \eqref{eq:rand_fiel_omeg} and \eqref{eq:rand_fiel_omeg'}, the restriction of the $\mathcal{V}$ to $\mathbb{Z}^3$ has the law of the vacant set of a $3$-dimensional Bernoulli line percolation diminished by an independent $3$-dimensional Bernoulli site percolation.

Let us denote for the moment, $\mathbb{P} = \mathbb{P}_{p_1} \times \cdots \times \mathbb{P}_{p_d}$.
In view of Lemma \ref{lemma:tube_2}, in the event that $|C(\omega_1)| < \infty$, we have that $\mathbb{P} (\omega \in \{0 \leftrightarrow \infty\}) = 0$ thus, in order to prove Theorem \ref{theo:conn_deca_2}, it suffices to show that there exists positive constants $\alpha$ and $\alpha'$ depending on $\mathbf{p}$ such that, for every large enough integer $k$,
\begin{equation}
\label{eq:poly_deca_rest}
\mathbb{P} (\omega \in \{0 \leftrightarrow \partial{B}^3(k) \text{ in $\mathbb{Z}^3$}\},\, |C(\omega_1)| < \infty) \geq \alpha' k^{-\alpha}.
\end{equation}

\medskip

\noindent \textbf{A word about the notation}

In order to obtain equation \eqref{eq:poly_deca_rest} we will restrict ourselves to the $\mathbb{Z}^3$ lattice as discussed above.
For this reason, we use a particular notation throughout this section.
Sites will be denoted by $(x,y,z) \in \mathbb{Z}^3$ and the $2$-dimensional coordinate planes will be by definition:
\begin{equation*}
\begin{split}
 \mathcal{P}_1 & = \{(x,y,0);~y,z \in \mathbb{Z}\}; \\
 \mathcal{P}_2 & = \{(x,0,z);~x,z \in \mathbb{Z}\}; \\
 \mathcal{P}_3 & =  \{(0,y,z);~x,y \in \mathbb{Z}\}.
\end{split}
\end{equation*}
We change the role of the symbols $\pi_i$ that will now denote the orthogonal projections from $\mathbb{Z}^3$ into $\mathcal{P}_i$ for $i=1,2,3$.
The reader should remember that we are considering a model in $3$ dimensions however with $d$ parameters, $p_1, \ldots, p_d$ from which we generate $\mathbf{p}'=(p_1,p_2,p_3)$ and $p''= p_4 \cdots p_d$.
We will also change the role of the symbols $\omega_i$ to refer now to $2$-dimensional percolation processes defined in $\mathcal{P}_i$ with law $\mathbb{P}_{p_i}$.
We denote by $\omega' (v) = \omega_1 (\pi_1(v)) \omega_2(\pi_2(v)) \omega_3(\pi_3(v))$ the random element whose law in $\{0,1\}^{\mathbb{Z}^3}$ is given by $\mathbb{P}_{\mathbf{p}'}$.
We denote by $\omega''$ the random element whose law in $\{0,1\}^{\mathbb{Z}^3}$ is given by $\mathbb{P}_{p''}$.
Finally we write, for each $v \in \mathbb{Z}^3$, $\omega(v) = \omega'(v) \omega''(v)$.
\medskip

\noindent \textbf{Organization of the section}

The remainder of this section is organized as follows:
In subsection \ref{subsec:cro_eve} we construct a rescaled lattice whose sites are blocks of vertices from the original $\mathbb{Z}^3$-lattice.
Such a block is declared to be \emph{good} depending on the state of the $\omega''$ processes inside it and also on the occurrence of some crossings for the $\omega_2$ and $\omega_3$ processes restricted to a finite region around its projection into $\mathcal{P}_2$ and $\mathcal{P}_3$ respectively.
Provided that the side length of the blocks are chosen sufficiently large and that $p''$ is very large, we obtain regions in the block-lattice where we can find, with high probability, crossings of good blocks.
In subsection \ref{subsec:con_pat} we show that to each such crossing of good blocks there corresponds long paths in the original $\mathbb{Z}^3$-lattice that are simultaneously $\omega''$, $\omega_2$ and $\omega_3$-open.
In subsection \ref{subsec:poly_deca} we show how to use those long paths for building long finite $\omega$-open paths in $\mathbb{Z}^3$ which ultimately yields equation \eqref{eq:poly_deca}.

\subsection{Crossing events in a block lattice}
\label{subsec:cro_eve}

Let $R(n,m) \in \mathbb{Z}^2$ denote the following rectangular region:
\[
R(n,m) =\{(x,y) \in \mathbb{Z}^2; ~0 \leq x \leq n-1, ~0 \leq y \leq m-1\}.
\]
For $l,~ k \in \mathbb{Z}$ we define 
\[
R(n,m;k,l) = R(n,m)+\{(k,l)\}.
\]
When $n=m$ we may drop the index $m$ and write simply $R(n;k,l)$ for referring to $R(n,n;k,l)$.
A path of nearest-neighbors sites traversing $R(n,m;k,l)$ from its left side to its right side will be called a \emph{left-to-right crossing} in $R(n,m;k,l)$. 
Similarly we define a \emph{bottom-to-top crossing} in $R(n,m;k,l)$.
For a given $\eta \in \{0,1\}^{\mathbb{Z}^2}$ we say that a crossing is open if all of its sites are $\eta$-open.
We also define the following events
\begin{equation}
\label{eq:cro}
 \begin{split}
  \mathcal{A}(n,m;k,l) & = \{ \text{there is an open left-to-right crossing in}~ R(n,m;k,l) \}, \\
  \mathcal{B}(n,m;k,l) & = \{ \text{there is an open bottom-to-top crossing in}~ R(n,m;k,l) \},
 \end{split}
\end{equation}
 and we write, for instance, $\mathcal{B}(n,m)$ in place of $\mathcal{B}(n,m;0,0)$. Recalling that each $\mathcal{P}_i$ is isomorphic to $\mathbb{Z}^2$, we write $R_i (n,m;k,l)$ in order to refer to the analogue of the rectangles $R(n,m;k,l)$ that lie on $\mathcal{P}_i$ and $\mathcal{A}_i (n,m;k,l)$ and $\mathcal{B}_i (n,m;k,l)$ for the events  in $\{0,1\}^{\mathcal{P}_i}$ that are analogous to the crossing events appearing in \eqref{eq:cro}. 

We will use repeatedly the fact that, if $p>p_c(\mathbb{Z}^2)$ and $c=c(p)$ is a sufficiently large constant (\textit{e.g.}~greater than the correlation length for $\mathbb{P}_p$), then:
\begin{equation}
\label{eq:cro_log_2}
\lim_{n \to \infty} \mathbb{P}_p \left( \mathcal{B} \left( \lceil c\log n \rceil, n \right) \right) = 1.
\end{equation}
This fact is standard in $2$-dimensional Bernoulli site percolation
(for reader convenience we provide a proof in the Appendix \ref{app:cro_log}).

Let $\tilde{\Gamma}_n(0,0,0) = ([0,n-1)\times [0,n-1)\times [0,n-1)) \cap \mathbb{Z}^3$
and define 
\[
\tilde{\Gamma}_n(j,l,h)= \tilde{\Gamma}_n(0,0,0) + \{(jn, ln, hn)\}.
\] 
(see Figure \ref{fig:good_box} for a illustration of $\tilde{\Gamma}_n(3,3,3)$).
Note that $\tilde{\Gamma}_n(j,l,h)$ is a block with side length $n$ and satisfies
\[
\begin{split}
\pi_1 \left( \tilde{\Gamma}_n(j,l,h) \right) & = R_1 (n;jn,ln),\\
\pi_2 \left( \tilde{\Gamma}_n(j,l,h) \right) & = R_2 (n;jn,hn),\\
\pi_3 \left( \tilde{\Gamma}_n(j,l,h) \right) & = R_3 (n;ln,hn).
\end{split}
\] 
We define
\[ 
\Lambda_n = \{ \tilde{\Gamma}_n(j,l,h); ~j, l, h \in \mathbb{Z} \}
\]
and insert an edge between two blocks $\tilde{\Gamma}_n(j,l,h)$ and $\tilde{\Gamma}_n(j',l',h')$ whenever $|j - j'| + |l - l'| + |h - h'| = 1$, making thus $\Lambda_n$ isomorphic to the $\mathbb{Z}^3$-lattice.
We define also the event
\[
\mathcal{C}(n;j,l,h) = \{\text{$\eta \in \{0,1\}^{\mathbb{Z}^3}$; all the sites in $\tilde{\Gamma}_n(j,l,h)$ are $\eta$-open}\}.
\]

\begin{definition}
Given $(\omega_2, \omega_3, \omega'') \in \mathcal{P}_2 \times \mathcal{P}_3 \times \{0,1\}^3$, a block $\tilde{\Gamma}_n(j,l,h)$ is said to be \emph{good} if:
\begin{equation*}
\begin{split}
\omega_2 & \in \mathcal{A}_2 (2n,n;jn,hn) \cap \mathcal{B}_2 (n,2n;jn,hn),\\\omega_3 & \in \mathcal{A}_3 (2n,n;ln,hn) \cap \mathcal{B}_3 (n,2n;ln,hn), \\\omega''  & \in \mathcal{C}(n;j,l,h).
\end{split}
\end{equation*}
See figure \ref{fig:good_box}).
\end{definition}

For $c > 0$ and a nonnegative integer $k \in \mathbb{Z}$ let
\begin{equation*}
\begin{split}
& \tilde{R}_n (c \log k, k) =  \left\{ \tilde{\Gamma}_n \left( \lfloor j/2 \rfloor, ~\lceil j/2 \rceil,~ h \right) \in \Lambda_n;\; 0 \leq j \leq \lceil c \log k \rceil -1, ~ 0 \leq h \leq k-1 \right\}
\end{split} 
\end{equation*}
where, for a real number $a$, $\lceil a \rceil = \min\{n \in \mathbb{Z}; ~ n \geq a \}$ and $\lfloor a \rfloor = \max \{n \in \mathbb{Z}; ~ n \leq a\} $. 
When regarded as a sub-graph of $\Lambda_n$, $\tilde{R}_n(c\log k, k)$ is isomorphic to the rectangular region $R(\lceil c \log k \rceil, k ) \subset \mathbb{Z}^2$ (see Figure \ref{fig:good_box}).
This allows us to define the event
\[ \tilde{\mathcal{B}}_n (c\log k , k) = \left\{ \exists \text{ a bottom-to-top crossing of good boxes in} ~ \tilde{R}_n (c\log k, k) \right\} \]
wich are analogous to the event $\mathcal{B}(\rceil c \log{k}, k)$ defined above.

\begin{figure}[htb]
\centering
\includegraphics[width=0.8\textwidth]{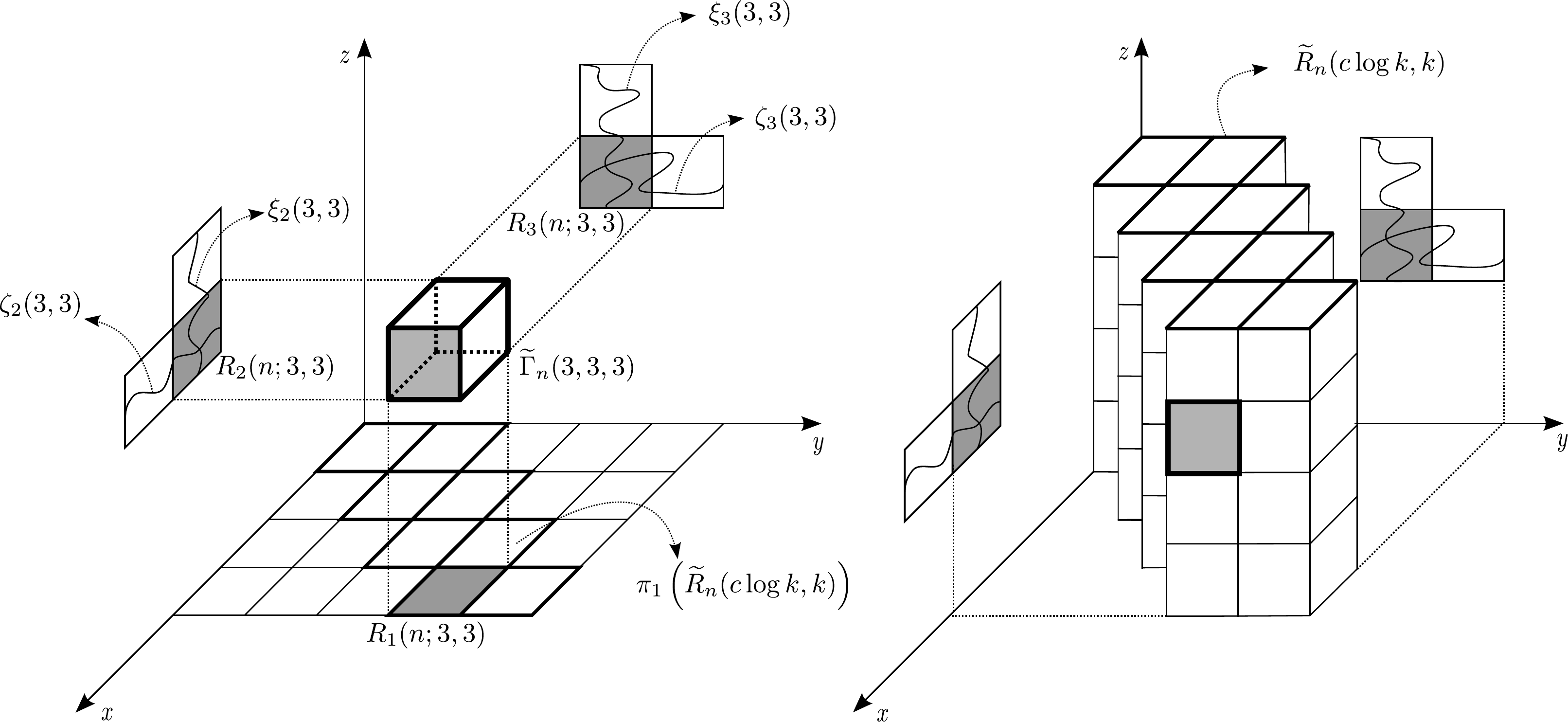}
\caption[Definition of a block to be good and of the subset $\tilde{R}_n(c\log k, k)$]{In this picture, the block $\tilde{\Gamma}_n(3,3,3)$ is a good block. In the left picture we see the definitions of the paths $\xi_i (3,3)$ and $\zeta_i (3,3)$ for $i =2, 3$. In the right we represent the set $\tilde{R}_n(c\log k, k)$.}
\label{fig:good_box}
\end{figure}

\begin{lemma}
\label{lemma:pat_blo_goo}
Let $p_2,~p_3 > p_c(\mathbb{Z}^2)$.
There exist $n = n(p_2, p_3) \in \mathbb{Z}_{+}$, $c > 0$, $\delta  >0$ and $p_{\bullet} \in (0,1)$ such that, if $p'' > p_{\bullet}$ then
\begin{equation}
\label{eq:Pp2xPp3}
\mathbb{P}_{p_2} \times \mathbb{P}_{p_3} \times \mathbb{P}_{p''} \left( \tilde{\mathcal{B}}_n (c \log k, k) \right) \geq \delta, \text{  for all $k \in \mathbb{Z}$}.
\end{equation}
\end{lemma}
\begin{proof}
Fix $p$ such that $p_i > p > p_c(\mathbb{Z}^2)$ for $i = 2,3$.
As an application of equation (\ref{eq:cro_log_2}) (with $k$ playing the role of $n$) we can choose constants $c = c(p) > 0$ and $\delta = \delta(p) > 0$ such that
\[ \mathbb{P}_{p} \left( \mathcal{B} \left( \lceil c \log k \rceil, k \right) \right) \geq \delta \]
for all integer $k \geq 0$.

Let $X = \{ X(j,h) \}_{(j,h) \in \mathbb{Z}^2}$ be the process on $\{0,1\}^{\mathbb{Z}^2}$ given by
\[ X(j,h) = \mathbf{1}_{\left\{\tilde{\Gamma}_n\left( \lfloor \frac{j}{2} \rfloor, \lceil \frac{j}{2} \rceil, h \right) ~\text{is good} \right\} }\, .\]

The definition of a block to be good depends only on the values of $\omega''$ inside it and on the restriction of the $\omega_2$ and $\omega_3$ processes to the projections of this block and of its neighboring blocks into the planes $\mathcal{P}_2$ and $\mathcal{P}_3$.
Thus, under $\mathbb{P}_{p_2} \times \mathbb{P}_{p_3} \times \mathbb{P}_{p''}$, $X$ is a translation invariant two-dependent percolation process on $\mathbb{Z}^2$.
If follows from \cite[Theorem 0.0]{liggett97} that there exists $u \in (0,1)$ such that,
\begin{equation}
\label{eq:pat_blo_goo_2}
\begin{array}{c}
\text{if \,\, $\mathbb{P}_{p_2} \times \mathbb{P}_{p_3} \times \mathbb{P}_{p''} \left( \tilde{\Gamma}_n( 0,0,0 ) ~\text{is good} \right) \geq u$, then $X$ dominates} \\
\text{stochastically a $2$-dimensional Bernoulli site percolation with parameter $p$.}
\end{array}
\end{equation}
Since $p_2$ and $p_3$ are larger than $p_c(\mathbb{Z}^2)$, by equation \eqref{eq:cro_log_2} there exists $n = n(p_2,p_3)$ such that,
\[
\min_{i=2,3} \mathbb{P}_{p_i} \left( \mathcal{B}_i (n,2n) \right) \geq {u}^{1/8}.
\]
Having fixed $n$, define $p_{\bullet} : = u^{1/2n^3}$.
Then, whenever $p''\geq p_{\bullet}$ we verify that 
\[
\mathbb{P}_{p''}\left( \mathcal{C}(n;j,l,h) \right) = {(p'')}^{n^3} \geq \sqrt{u}.
\]

By the Harris-FKG inequality (see \cite{Harris60} or \cite{Grimmett99}) and symmetry,
\begin{equation*}
\begin{split}
& \mathbb{P}_{p_2} \times \mathbb{P}_{p_3} \times \mathbb{P}_{p''} \left( \tilde{\Gamma}_n(0,0,0) ~\text{is good} \right) \geq \\  
& \left[ \mathbb{P}_{p_2} \left( \mathcal{B}_2 (n,2n) \right)\right]^{2} \times \left[ \mathbb{P}_{p_3} \left( \mathcal{B}_3 (n,2n) \right) \right]^{2} \times \mathbb{P}_{p''} \left( \mathcal{C}(n;j,l,h) \right) \geq u.
\end{split}
\end{equation*}
which in view of (\ref{eq:pat_blo_goo_2}) gives
\[ 
\mathbb{P}_{p_2} \times \mathbb{P}_{p_3} \times \mathbb{P}_{p''} \left( \tilde{\mathcal{B}}_n (c \log k, k) \right) \geq \mathbb{P}_{p} \left( \mathcal{B} \left( \lceil c \log k \rceil, k \right) \right) \geq \delta.
\]
\end{proof}

\subsection{Constructing paths from projections}
\label{subsec:con_pat}

For this section we fix the constants $c > 0$, $\delta>0$, $p_\bullet = p_\bullet (p_2,p_3) \in (0,1)$ and a nonnegative integer $n = n(p_2,p_3)$ so that equation \eqref{eq:Pp2xPp3} holds (see Lemma \ref{lemma:pat_blo_goo}). 
Since $n$ is fixed it will be omitted in the subscripts in $\tilde{\Gamma}_n$, $\tilde{\mathcal{B}_n}$, $\tilde{R}_n$, $\Lambda_n$ and others. 

Lemma \ref{lemma:pat_blo_goo} assures that the probability of existence of paths of good blocks in the set $\tilde{R}(c \log k, k)$ is bounded from below uniformly as $k$ increases.
In order to easy the notation, let us not distinguish between $\tilde{R}(c \log k, k) \subset \Lambda$ and $\bigcup_{j,l,h} \tilde{\Gamma}(j,l,h) \subset \mathbb{Z}^3$ where the last union is taken over the set $\left\{(j,l,k); ~0\leq j \leq \lceil c\log k \rceil, ~0 \leq h \leq k, ~l=j ~\text{or}~ l=j+1 \right\}.$

The main goal of the present section is to show that, to each bottom-to-top crossing of good blocks in $\tilde{R}(c \log k, k)$ there corresponds a path $\gamma = \{v_0, v_1, \ldots, v_r \}$ of sites of $\mathbb{Z}^3$ having the following properties:
\begin{itemize}
\item $\gamma$ is contained in $\tilde{R}(c \log k, k)$.
\item $v_0 \in \{(x,y,z) \in \tilde{R}(c \log k, k); z = 0 \}$ and $v_r \in \{(x,y,z) \in \tilde{R}(c \log k, k); z = (k-1)n \}$, \emph{i.e.}~$\gamma$ crosses $\tilde{R}(c \log k, k)$ from bottom to top.
\item $\omega_i (v_j) = 1$ and $\omega''(v_j)=1$ for all $i = 2$ and $j \in \{0, \ldots, r\}$.
\end{itemize}

In order to prove the existence of such a path we present, as a lemma, a procedure that enables us to find paths in $\mathbb{Z}^3$ having their projections into $\mathcal{P}_2$ and $\mathcal{P}_3$ contained in some crossings of rectangles lying in these coordinate planes.
We begin introducing some more notation.

For a site $v = (x,y,z) \in \mathbb{Z}^3$ we define $h(z) = z$ the value of its third coordinate that we call the \emph{height} of $v$.
For two sites $v = (x,0,z) \in \mathcal{P}_2$ and $w = (0,y,z) \in \mathcal{P}_3$ with $h(v)= h(w) = z$ we define
\begin{equation}
 \label{eq:prod_site}
v \times w = (x,0,z) \times (0,y,z) = (x,y,z)
\end{equation}
In words, $v \times w$ is the unique site having $v$ and $w$ as projections onto $\mathcal{P}_2$ and $\mathcal{P}_3$ respectively.
Let
\begin{equation}
 \label{eq:path_gamm}
\gamma = \{ v_0, v_1, \ldots, v_m\} \subset \mathbb{Z}^3
\end{equation}
be a path.
We define its \emph{height variation} as being
\begin{equation}
 \label{eq:vari_path}
h(\gamma) = h(v_m) - h(v_0).
\end{equation}
For any $0 \leq k \leq m$ we denote
\begin{equation}
 \label{eq:path_stop}
\gamma^{(k)} = \{v_0, \ldots, v_k \}
\end{equation}
the path $\gamma$ stopped at its $k$-th step.
Denoting for any $k \in \mathbb{Z}$
\begin{equation}
 \label{eq:hitt_time}
\tau_k = \tau_k (\gamma) =  \inf \{j \geq 0; ~ h(v_j) = k \}
\end{equation}
we can define the path stopped at the first time it hits height $k$ by:
\begin{equation}
 \label{eq:path_mini_k}
 \gamma \wedge k = \gamma^{(\tau_k)}.
\end{equation}
In case the infimum in equation (\ref{eq:hitt_time}) is taken over an empty set we define $\gamma \wedge k = \gamma$.
For a path $\gamma$ as in (\ref{eq:path_gamm}) we define its \emph{reversal} as being the path
\begin{equation}
 \label{eq:path_reve}
\overline{\gamma} = \{v_m, v_{m-1}, \ldots, v_0\}.
\end{equation}
If we now take $\gamma$ as in (\ref{eq:path_gamm}) and another path 
\begin{equation}
\label{eq:path_gamm_prim}
\gamma'=\{w_0, \ldots, w_{m'} \}
\end{equation}
having $w_0 = v_m$ we define their \emph{concatenation} by:
\begin{equation}
 \label{eq:conc_path}
\gamma * \gamma' = \{v_0, \ldots, v_m, w_1, \ldots, w_{m'} \}.
\end{equation}
\begin{remark}
The concatenation of two paths does not need to be a self-avoiding path.
From now on we will relax the definition of a path and also refer to concatenations as paths.
Noting however that, given a path that is not self-avoiding, it is always possible to extract a self-avoiding path contained in it and having the same starting and end vertices, the reader should be convinced that this slight sloppiness in the writing does not affect the results below.
\end{remark}
Now, if $\gamma$ and $\gamma'$ given by (\ref{eq:path_gamm}) and (\ref{eq:path_gamm_prim}) intersect themselves however not necessarily in their endpoints we define their \emph{juxtaposition} $\gamma \circ \gamma'$ as being the unique path described as follows:
It starts at $v_0$, goes along $\gamma$ until it first hits $\gamma'2$ and then, from that point on, it goes along $\gamma'$ until its final point $w_{m'}$.

We say that two paths $\gamma$ and $\gamma'$ given as in (\ref{eq:path_gamm}) and (\ref{eq:path_gamm_prim}) given are \emph{compatible} if
\begin{equation}
 \label{eq:comp_1}
\gamma \subset \mathcal{P}_2 \text{ and } \gamma' \subset \mathcal{P}_3
\end{equation}
\begin{equation}
 \label{eq:comp_2}
 h(v_0) = h(w_0) \text{ and } h(v_m) = h(w_{m'})
\end{equation}
\begin{equation}
 \label{eq:comp_3}
\begin{split}
 h(\gamma^{(k)}) h(\gamma) \geq 0 \text{ for all } 0 \leq k \leq m \text{ and } 
 h(\gamma'^{(l)}) h(\gamma') \geq 0 \text{ for all } 0 \leq l \leq m'
\end{split}
\end{equation}
\begin{equation}
 \label{eq:comp_4}
\tau_h(\gamma) = m \text{ and } \tau_h (\gamma') = m' \text{, where } h = h(\gamma) = h(\gamma').
\end{equation}
Condition (\ref{eq:comp_2}) states that both $\gamma$ and $\gamma'$ start and finish at the same height, and it implies that $h(\gamma) = h(\gamma')$.
Condition (\ref{eq:comp_3}) requires that as one moves forward along the paths, the variation does not change signs, meaning that the paths will always lie above or beneath their initial height.
Finally, condition (\ref{eq:comp_4}) guarantees that $\gamma$ and $\gamma'$ first hit the final height at their respective end points $v_m$ and $w_{m'}$.

The point in defining the notion of compatible paths is that whenever $\gamma$ and $\gamma'$ are compatible it is possible to find a path in $\mathbb{Z}^3$ having $\gamma$ and $\gamma'$ as its projections into $\mathcal{P}_2$ and $\mathcal{P}_3$ respectively. This is the content of the following lemma:

\begin{lemma}
\label{lemma:geo}
Let $\gamma$ and $\gamma'$ be given as in \eqref{eq:path_gamm} and \eqref{eq:path_gamm_prim} be two compatible paths.
There is a path $\gamma \times \gamma' \subset \mathbb{Z}^3$ starting at $v_0 \times w_0$ and ending at $v_m \times w_{m'}$ satisfying:
\begin{equation}
\label{eq:proj_gamm_gamm_prim}
\pi_2(\gamma' \times \gamma) = \gamma \,\,\text{ and } \,\,\pi_3(\gamma \times \gamma') = \gamma'.
\end{equation}
\end{lemma}
Although some readers may find this result extremely intuitive, we give a proof to it at the Appendix \ref{lemma:geo}.

\begin{remark}
There can be more than one path connecting $v_0 \times w_0$ and $v_m \times v_m$ and satisfying \eqref{eq:proj_gamm_gamm_prim}.
So whenever we write $\gamma \times \gamma'$ we are referring to one of those paths arbitrarily selected.
\end{remark}

When $\mathcal{A}_i (2n, n; j,h) \cap \mathcal{B}_i (n, 2n; j,h)$ occurs we denote by $\xi_i(j,h)$ a bottom-to-top $\omega_i$-open crossing in $R_i(n,2n;j,h)$ arbitrarily selected among the possible choices (for instance left-most among such crossings).
We also denote by $\zeta_i(j,h)$ a left-to-right $\omega_i$-open crossing in $R_i(2n,n;j,h)$ also arbitrarily selected (for instance the lowest one).
We write $\left( \xi_i(j,h) \right)_0$ and $\left(\zeta_i(j,h)\right)_0$ in order to refer to the starting points of those paths.

Now suppose that $\tilde{\Gamma}(j,l,h)$ and $\tilde{\Gamma}(j',l',h')$ are good blocks that are neighbors in the graph $\Lambda$. 
We will use Lemma \ref{lemma:geo} in order to construct paths contained in the union of these blocks, joining  $\left(\xi_2(j,h)\right)_0 \times \left(\xi_3(l,h)\right)_0$ to $\left(\xi_2(j',h')\right)_0 \times \left(\xi_3(l',h')\right)_0$ and, in addition, having their projections into $\mathcal{P}_i$ (for $i = 2,3$) contained in the union of the corresponding $\xi_i$, $\zeta_i$ paths.

\begin{lemma}
\label{lemma:geo_2}
Assume that $\tilde{\Gamma}(j,l,h)$ and $\tilde{\Gamma}(j',l',h')$ are good neighboring blocks in $\Lambda$.
Then there is a path $\gamma = \{v_0, v_1, \ldots, v_m \} \subset \mathbb{Z}^3$ satisfying:
\begin{enumerate}
 \item  $\gamma \subset \tilde{\Gamma}(j,l,h) \cup \tilde{\Gamma}(j',l',h')$;
 \item  $v_0 = \left(\xi_2(j,h)\right)_0 \times \left( \xi_3(l,h) \right)_0$ and $v_m = \left(\xi_2(j',h')\right)_0 \times \left( \xi_3(l',h')\right)_0$;
 \item  $\pi_2 (\gamma) \subset \zeta_2(j,h) \cup \xi_2(j,h) \cup \zeta_2(j', h') \cup \xi_2(j',h')$;
 \item  $\pi_3 (\gamma) \subset \zeta_3(l,h) \cup \xi_3(l,h) \cup \zeta_3(l', h') \cup \xi_3(l',h')$.
\end{enumerate}
In particular, any site $v \in \gamma$ has $\omega''(v)=1$ and, in addition, as a consequence of the items $3$ and $4$ we have $\pi_2(v)$ is $\omega_2$-open and $\pi_3(v)$ is $\omega_3$-open.
\end{lemma}
\begin{proof}
Since $\tilde{\Gamma}(j,l,h)$ and $\tilde{\Gamma}(j',l',h')$ are neighboring boxes we have that $|j'-j| + |l - l'| + |h - h'| = 1$. 
We will split the proof into six cases (each one corresponding one of the indices changing $\pm 1$ units) and use Lemma \ref{lemma:geo} in each of those cases. 
We only prove the cases $h'-h = \pm 1$ and $j'- j = \pm 1$, the remaining cases $l'-l = \pm 1$ are analogous.

 \begin{enumerate}

\item $h' - h$ = 1 (the basic strategy is depicted in Figure \ref{fig:const_path_1})

For convenience let us fix $h=0$, $j = j' = 0$ and $ l = l' = 0$. Since $\xi_2(0,0)$ and $\xi_3(0,0)$ are respectively bottom-to-top crossings of $R_2(n,2n)$ and $R_3(n,2n)$ they are compatible. 
By Lemma \ref{lemma:geo} we can pick a path $\xi = \xi_2(0,0) \times \xi_3(0,0) \subset \mathbb{Z}^3$ starting at $\left(\xi_2(0,0)\right)_0 \times \left(\xi_3(0,0)\right)_0$  and having $\pi_i (\xi) = \xi_i (0,0)$ for $i = 2,3$. 
In particular $\xi$ is contained in $\tilde{\Gamma}(0,0,0) \cup \tilde{\Gamma} (0,0,1)$.

Now, for $i = 2,3$ let $\beta_i (0,1) = \overline{\xi_i(0,0)} \circ \overleftrightarrow{\zeta_i (0,1)} \circ \overline{\xi_i(0,1)}$ be the path in $\mathcal{P}_i$ defined the following way: First start at the final point of $\xi_i(0,0)$, then go down along its reversal $\overline{\xi_i(0,0)}$ until first hitting $\zeta_i (0,1)$. 
After hitting $\zeta_i (0,1)$ go along this path in the appropriate sense in order to hit the path $\xi_i(0,1)$. 
Note that either $\zeta_i(0,1)$ or its reversal should be taken in order to hit $\xi_i(0,1)$. 
Finally, after hitting $\xi_i(0,1)$ take its reversal $\xi_i(0,1)$ until getting to its starting point $\left( \xi_i(0,1) \right)_0$.
\begin{remark}
The arrow is placed on the top of $\zeta_i (0,1)$ in order to indicate that one should goes along either $\zeta_i (0,1)$ or $\overline{\zeta_i (0,1)}$ depending on which one of these paths will lead to $\xi_i(0,1)$.
We prefer not to give a formal definition and trust that this description is enough for making the construction clear.
\end{remark}

Note that $\beta_2 (0,1)$ and $\beta_3 (0,1)$ are top-to-bottom crossings of $R_2(n;0,1)$ and $R_3(n;0,1)$ then they are compatible paths and by Lemma \ref{lemma:geo} we can pick a path $\beta = \beta_2(0,1) \times \beta_3(0,1) \subset \tilde{\Gamma}(0,1)$ connecting the ending point of $\xi$ to the site $\left( \xi_2(0,1) \right)_0 \times \left( \xi_3(0,1) \right)_0$ and having $\pi_i(\beta) \subset \xi_i (0,0) \cup \zeta_i (0,1) \cup \xi_i(0,1)$ for $i=2,3$.

Let us define $\gamma = \xi * \beta$. Then this path starts at $\left( \xi_2(0,0) \right)_0 \times \left( \xi_3(0,0) \right)_0$ finishes at $\left(\xi_2(0,1) \right)_0 \times \left(\xi_3(0,1)\right)_0$. 
The properties $1$, $3$ and $4$ in the statement are satisfied since they hold for both $\xi$ and $\beta$.

\begin{figure}[htb]
\centering
\includegraphics[width=0.6\textwidth]{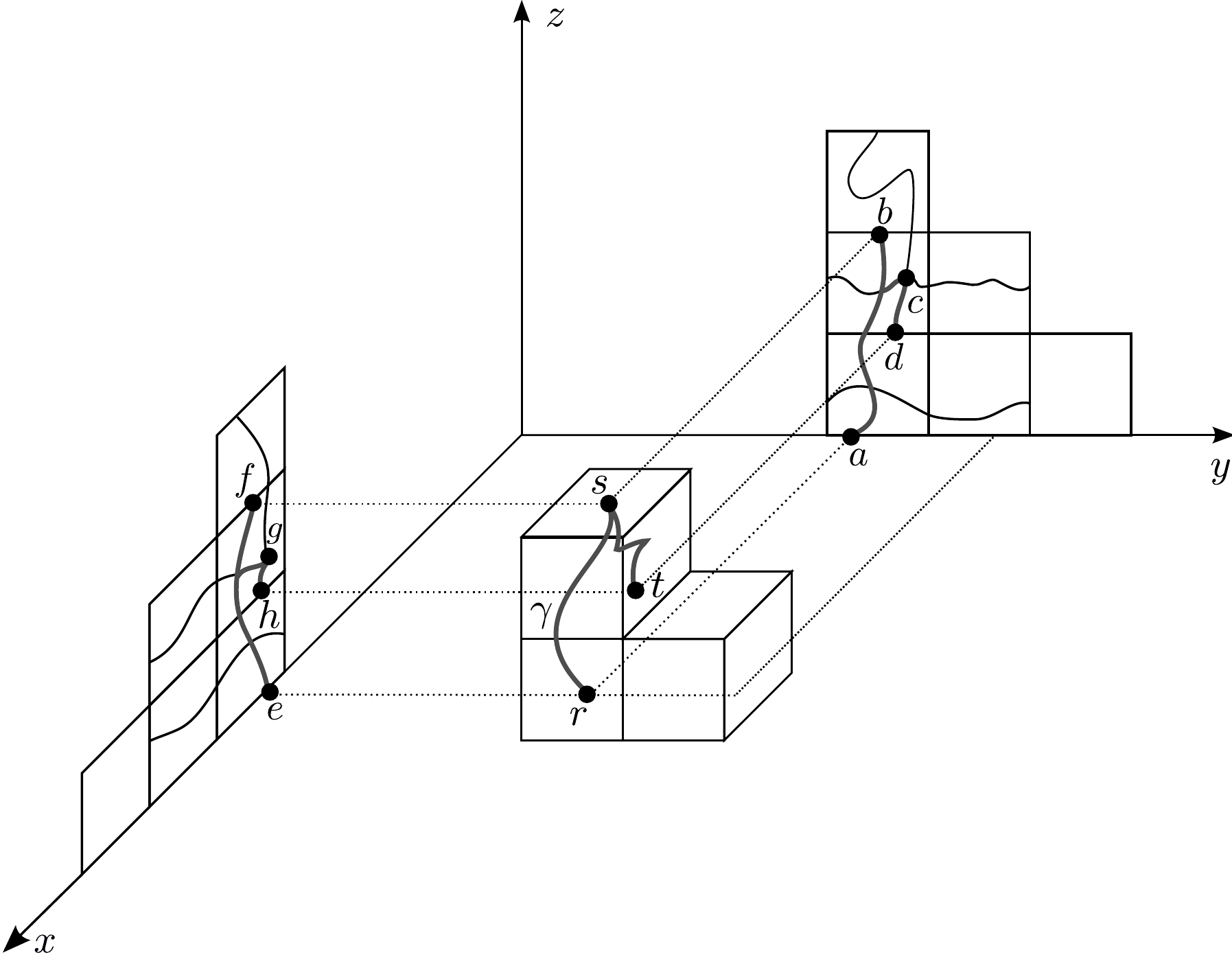}
\caption[Open paths in neighboring good blocks, case $h'-h =1$]{The path $\xi_2(0,0)$ joins $a$ to $b$. The path $\xi_3(0,0)$ joins $e$ to $f$. 
The path $\beta_2(1,0)$ starts at $b$, passes through $c$ and ends at $d$.
The path $\beta_3(1,0)$ starts at $f$, passes through $g$ and ends at $h$. 
The path $\gamma = (\xi_2 \times \xi_3) * (\beta_2 \times \beta_3) $ starts at $r$, passes through $s$ and ends at $t$.
It is contained in $\tilde{\Gamma}(j,l,0) \cup \tilde{\Gamma}(j,l,1)$.}
\label{fig:const_path_1}
\end{figure}
\item $h' - h = - 1$.

By the previous case if we interchange the roles of $h$ and $h'$ we can pick a path satisfying the properties $1$, $3$ and $4$, however starting at $\left(\xi_2(j',h')\right)_0 \times \left( \xi_3(l',h')\right)_0$ and finishing at $\left(\xi_2(j,h)\right)_0 \times \left( \xi_3(l,h) \right)_0$.
The reversal of this path satisfy all the required properties.

\item $j'-j = 1$ (the basic strategy is depicted in Figure \ref{fig:const_path_2})

In order to simplify the notation we fix $j=0$.
Let $\alpha_2 = \xi_2(0,0) \circ \zeta_2(0,0) \circ (\xi_2 (1,0) \wedge n) \subset \mathcal{P}_2$ be the following path: Start at $\left( \xi_2(0,0) \right)_0$ and go along $\xi_2(0,0)$ until it hits $\zeta_2(0,0)$. 
After hitting $\zeta_2(0,0)$ go along this path until hitting $\xi_2 (1,0)$. 
Finally go along $\xi_2(1,0)$ up to height $n$. Note that $\alpha_2$ is bottom-to-top crossing of the rectangle $R_2(2n,n;0,0)$.

Define now $\alpha_3 = \xi_3(0,0) \wedge n$. 
Then $\alpha_3$ is the bottom-to-top crossing of the rectangle $R_3(n;0,0)$ that starts at $\left(\xi_3(0,0)\right)_0$, goes along $\xi_3(0,0)$ up to the time it first hits height $n$. 

Since $\alpha_2$ and $\alpha_3$ are crossings of blocks with same hight they are compatible.
We can apply Lemma \ref{lemma:geo} in order to pick a path $\alpha = \alpha_2 \times \alpha_3$ that starts at $\left(\xi_2(0,0) \right)_0 \times \left(\xi_3(0,0) \right)_0$, goes up to height $n$ and that has projections $\pi_2 (\alpha) = \alpha_2 \subset \xi_2(0,0) \cup \zeta_2(0,0) \cup \xi_2(1,0)$ and $\pi_3 (\alpha) = \alpha_3 \subset \xi_3(0,0)$. 
In particular $\alpha$ is contained in $\tilde{\Gamma}(0,0,0) \cup \tilde{\Gamma}(1,0,0)$.

Let us now define $\beta_3 = \overline{\xi_3 (0,0) \wedge n}$ which is the path starting at the ending point of $\xi_3 (0,0) \wedge n$ and going along its reversed path $\overline{\xi_3(0,0)}$ until it hits $(\xi_3(0,0))_0$.
We also define $\beta_2 = \overline{\xi_2(1,0) \wedge n}$ to be the top-to-bottom crossing of $R_2(n,n;1,0)$ that starts at the ending point of $\xi_2 (0,1) \wedge n$ and goes along $\overline{\xi_2(0,1)}$ all the way down to the site $\left(\xi_2(0,1)\right)_0$. 

Note that $\beta_2$ and $\beta_3$ are compatible and so, once more, Lemma \ref{lemma:geo} enables us to select a path $\beta = \beta_2 \times \beta_3$ connecting the ending point of $\alpha$ to the site $\left( \xi_2(1,0)\right)_0 \times \left(\xi_3 (1,0) \right)_0$ and having $\pi_2(\beta) = \beta_2 \subset \xi_2(1,0)$ and $\pi_3(\beta) = \beta_3 \subset \xi_3(0,0) \cup \zeta_3 (0,0) \cup \xi_3(1,0)$. In particular $\beta \subset \tilde{\Gamma}(0,0,0) \cup \tilde{\Gamma}(1,0,0)$.

Then the concatenation $\gamma = \alpha * \beta$ is a path satisfying the properties $1$ to $4$ above.

\item $j'- j = -1$

Interchange the roles of $j$ and $j'$, use the previous case and then reverse the obtained path.
\end{enumerate}
\end{proof}

\begin{figure}[htb]
\centering
\includegraphics[width=0.6\textwidth]{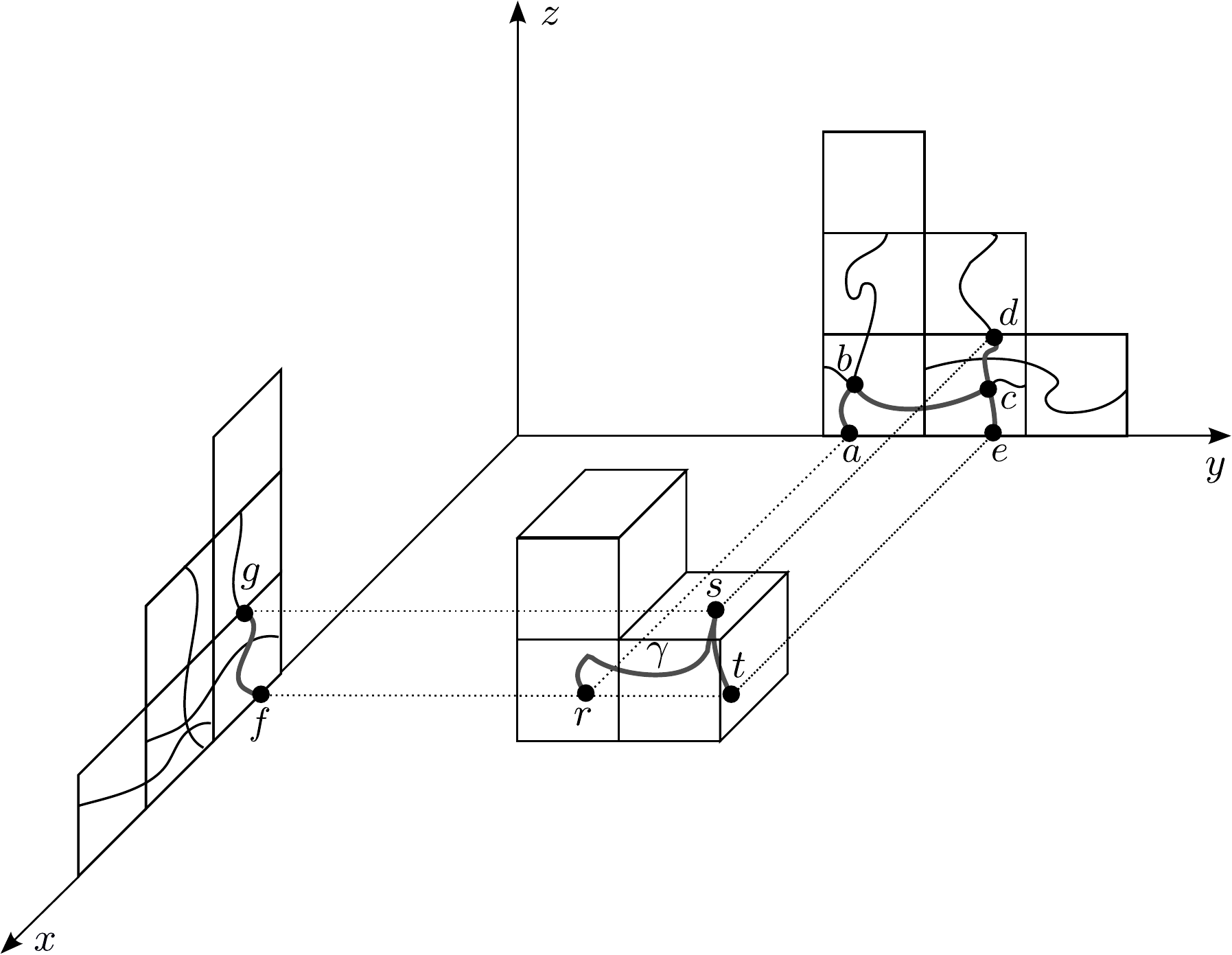}
\caption[Open paths in neighboring good blocks, case $j'-j =1$]{The path $\alpha_2$ starts at $a$, passes through $b$ and $c$, and ends at $d$. 
The path $\alpha_3$ starts at $f$ and ends at $g$. 
The path $\beta_3$ is simply its reversal, starting at $g$ and ending at $f$.
$\beta_2$ joins $d$ to $e$. 
The path $\gamma = (\alpha_2 \times \alpha_3)*(\beta_2 \times \beta_3)$ starts at $r$, passes through $s$ and ends at $t$. 
It is contained in $\tilde{\Gamma}(0,l,h) \cup \tilde{\Gamma}(1,l,h)$.}
\label{fig:const_path_2}
\end{figure}

\begin{lemma}
\label{cor:geo_2_2}
If $\tilde{\mathcal{B}}(c \log k, k)$ occurs and all sites in $\pi_1 \left( \tilde{R} (c \log k, k) \right)$ are $\omega_1$-open then there is a $\omega$-open path $\gamma = \{v_0, v_i, \ldots , v_r \} \subset \tilde{R}(c \log k, k)$ such that $h(v_0) = 0$ and $h(v_m) = (k-1)n$.
\begin{proof}
Recall that $\tilde{\mathcal{B}}(c \log k, k)$ is the event that there is a path 
\[ \tilde{\gamma} = \left\{ \tilde{\Gamma}(j_0, l_0 ,h_0),\ldots, \tilde{\Gamma}(j_m,l_m,h_m) \right\} \subset \Lambda \] crossing $R(c \log k, k)$.
We apply Lemma \ref{lemma:geo_2} to each pair $\tilde{\Gamma}(j_i,l_i,h_i)$ and $\tilde{\Gamma}(j_{i+1},l_{i+1},h_{i+1})$ that are neighboring blocks in $\tilde{\gamma}$ in order to find paths $\gamma_i$ contained in $\tilde{\Gamma}(j_i,l_i,h_i) \cup \tilde{\Gamma}(j'_{i+1},l'_{i+1},h'_{i+1})$; starting at $\left(\xi_2(j_i,h_i)\right)_0 \times \left( \xi_3(l_i,h_i) \right)_0$ and finishing at $\left(\xi_2(j_{i+1},h_{i+1})\right)_0 \times \left( \xi_3(l_{i+1},h_{i+1})\right)_0$ that have the projections into $\mathcal{P}_2$ and $\mathcal{P}_3$ being $\omega_2$ and $\omega_3$-open respectively.
Since the ending point of each of the $\gamma_i$ is the starting point of each of $\gamma_i$ we can concatenate them all obtaining a path $\gamma = \gamma_0 * \cdots * \gamma_m$ satisfying that:
\begin{enumerate}
 \item $\gamma$ is contained in $\bigcup_{i=0}^{m} \tilde{\Gamma}(j_i, l_i, h_i)$;
 \item $h(v_0) = n h_0 = 0$ and $h(v_r) = n h_m = (k-1)n$;
 \item all sites in $\pi_2(\gamma)$ and $\pi_3(\gamma)$ are $\omega_2$ and $\omega_3$-open respectively.
 \item all the sites in $\gamma$ are $\omega''$-open.
 \end{enumerate}

By the condition $1$) the fact that $\tilde{\gamma}$ is a crossing of $\tilde{\mathcal{B}}(c\log k, k)$ implies that $\gamma$ is contained in $\tilde{R} (c \log k, k)$, and since all sites in $\pi_1 \left( \tilde{R} (c \log k, k) \right)$ are $\omega_1$-open the projection of $\gamma$ into $\mathcal{P}_1$ is also composed of $\omega_1$-open sites.
In view of $3$) and $4$), we have that $\gamma$ is $\omega$-open which, combined with $2$) finishes the proof.
\end{proof}
\end{lemma}

\begin{remark}
\label{rmk:geo_2}
The same strategy of concatenating paths as in the proof of Lemma \ref{cor:geo_2_2} can be used to show that for any infinite path of good blocks in $\Lambda_n$, there corresponds an infinite path $\gamma$ in $\mathbb{Z}^3$ whose sites are $\omega''$-open and have their projections into $\mathcal{P}_2$ and $\mathcal{P}_3$ being $\omega_2$ and $\omega_3$ open respectively.
\end{remark}

\subsection{Proof of Theorem \ref{theo:conn_deca_2}}
\label{subsec:poly_deca}

\begin{proof}[Proof of Theorem \ref{theo:conn_deca_2}]
As in the previous section we fix constants $c> 0$, $\delta>0$, $p_\bullet = p_\bullet (p_2, p_3) \in (0,1)$ and a nonnegative integer $n = n(p_2,p_3)$ so that (see Lemma \ref{lemma:pat_blo_goo}), $p'' \geq p_{\bullet}$,
\begin{equation}
\label{eq:p_p_btilde}
\mathbb{P}_{p_2} \times \mathbb{P}_{p_3} \times \mathbb{P}_{p''}\left(\tilde{\mathcal{B}}(c \log k, k) \right) \geq \delta
\end{equation}
for all positive integers $k \geq 1$.
We assume that $k$ is big enough so that
\begin{equation*}
\begin{split}
& \mathbb{P} \left( \left\{ \mathbf{0} \leftrightarrow \partial{B^3 \left( k \right) \text{ in $\mathbb{Z}^3$}} , |C_1(\omega_1)| < \infty \right\} \right) \geq \mathbb{P}_{\mathbf{p}} \left( \left\{ \mathbf{0} \leftrightarrow \partial{B^3 \left( (k-1) n \right)} , |C_1(\omega_1)| < \infty \right\} \right).
\end{split}
\end{equation*}

Let us now define the following events:
\begin{equation*}
 \mathcal{C}_2 = \left\{ \text{all sites}~(x,0,0) \in \mathcal{P}_2~\text{such that}~ 0 \leq y \leq n (c \log{k}+1)/2 ~\text{are}~\omega_2 \text{-open} \right\}
\end{equation*}
\begin{equation*}
 \mathcal{C}_3 = \left\{ \text{all sites}~(0,y,0) \in \mathcal{P}_3~\text{such that}~ 0 \leq x \leq n (c \log{k}+1)/2 ~\text{are}~\omega_3 \text{-open} \right\}
\end{equation*}
\begin{equation*}
 \mathcal{D}_1 = \left\{ \text{all sites in } \pi_1 \left( \tilde{R}(c \log k , k) \right)\text{ are } \omega_1 \text{-open} \right\}
\end{equation*}
\begin{equation*}
\mathcal{D}'' = \left\{ \text{all sites in } \pi_1 \left( \tilde{R}(c \log k , k) \right)\text{ are } \omega'' \text{-open} \right\}
\end{equation*}
\begin{equation*}
 \mathcal{E}_1 = \left\{ \text{all sites of $\mathcal{P}_1$ lying at $l_{\infty}$-distance $1$ from $\pi_1\left(\tilde{R}(c \log k, k) \right)$ are $\omega_1$-closed} \right\}.
\end{equation*}

Using the fact that $\mathcal{D}_1$ and $\mathcal{E}_1$ are independent and a simply counting of the number of sites in $\tilde{R}(c \log{k},k)$ and its outer boundary, we have
\begin{equation}
\label{eq:pol_dec_1}
\mathbb{P}_{p_1} \left( \mathcal{D}_1 \cap \mathcal{E}_1 \right) = \mathbb{P}_{p_1}(\mathcal{D}_1) \mathbb{P}_{p_1}(\mathcal{E}_1) \geq {p_1}^{n^2(c \log k + 1)} (1-p_1)^{4n(c \log k + 1)}.
\end{equation}
Note also that for $i=2,3$,
\begin{equation}
\label{eq:PpiCi}
\mathbb{P}_{p_i} (\mathcal{C}_i) \geq p_i^{n(c\log{k} +1)/2}
\end{equation}
and
\begin{equation}
\label{eq:Pp''C''}
\mathbb{P}_{p''}(\mathcal{D}'') \geq (p'')^{n^2(c \log{k}+1)}.
\end{equation}

Now, if $\omega_1 \in \mathcal{D}_1$ and $(\omega_2, \omega_3, \omega'')\in \tilde{\mathcal{B}}(c \log k, k)$ then by Lemma \ref{cor:geo_2_2} there exists a $\omega$-open path starting at a (random) site $v_0$ in $\mathcal{P}_1$ and finishing at a site $v_r$ in  $\partial {B^{3}\left((k-1) n \right)}$.
Then, in order to have the origin connected to $\partial{B^3((k-1)n)}$ it is enough to guarantee that it is connected to $v_0$ by an $\omega$-open path contained in $\pi_1(\tilde{R}(c\log{k}, n)$.
This can be accomplished by simply requiring further that $\omega_2 \in \mathcal{C}_2$, $\omega_3 \in \mathcal{C}_3$ and $\omega'' \in \mathcal{D}''$ since this would garantee that all sites in $\pi_1 \left(\tilde{R}(c \log k, k)\right)$ are $\omega$-open. 
Thus we have concluded that  $\left\{ \omega \in \left\{ \mathbf{0} \leftrightarrow \partial{B^3 \left( (k-1) n \right)} \text{ in $\mathbb{Z}^3$} \right\} \right\}$ contais
\begin{equation*}
 \{\omega_1 \in \mathcal{D}_1, \omega'' \in \mathcal{D}'',\; \omega_2 \in  \mathcal{C}_2,\; \omega_3 \in \mathcal{C}_3,\; (\omega_2, \omega_3, \omega'') \in \tilde{\mathcal{B}}(c \log k, k) \}.
\end{equation*}

In addition, if the event $\{ \omega_1 \in \mathcal{E}_1\}$ occurs, then the $\omega_1$-open connected component at the origin in $\mathcal{P}_1$ is finite, \emph{i.e}.\ $|C_1(\omega_1)| < \infty$.
We can then conclude that
\begin{equation}
\label{eq:pol_dec_2}
\begin{split}
& \mathbb{P} \left( \omega \in \left\{ \mathbf{0} \leftrightarrow \partial{B^3 \left( (k-1) n \right)} \text{ in $\mathbb{Z}^3$} \right\}, |C_1(\omega_1)| < \infty  \right) \geq \\
& \mathbb{P}_{p_1} \left( \mathcal{D}_1 \cap \mathcal{E}_1 \right) \mathbb{P}_{p_2} (\mathcal{C}_2) \mathbb{P}_{p_3} (\mathcal{C}_3) \mathbb{P}_{p''} (\mathcal{D}'') \mathbb{P}_{p_2} \times \mathbb{P}_{p_3} \times \mathbb{P}_{p''} \left( \tilde{\mathcal{B}}(c \log k, k) \right)
\end{split}
\end{equation}
were we used the Harris-FKG inequality in order to decouple the events defined in terms of $\omega_2$, $\omega_3$ and $\omega''$.

Now plugging equations (\ref{eq:pol_dec_1}) \eqref{eq:PpiCi} and \eqref{eq:Pp''C''} into equation (\ref{eq:pol_dec_2}) we get:
\begin{equation*}
\begin{split}
&\mathbb{P} \left( \omega \in \left\{ \mathbf{0} \leftrightarrow \partial{B^3 \left( (k-1) n \right)} \text{ in $\mathbb{Z}^3$}\right\}, |C| < \infty  \right)  \geq \\
& \delta ~{p_1}^{n^2(c \log k + 1)} {(1-p_1)}^{4n(c \log k + 1)}~ p_2^{n(c \log k + 1)/2} ~p_3^{n(c \log k + 2)/2}~ (p'')^{n^2 (c \log{k}+1)} = \\
& \alpha'(\mathbf{p}) k^{-\alpha(\mathbf{p})},
\end{split}
\end{equation*}
where the constants $\alpha'$ and $\alpha$ depend on $\mathbf{p}$.

\end{proof}

\subsection{Proof of Theorem \ref{theo:pha_tra_2}}

In this section we prove Theorem \ref{theo:pha_tra_2}.
We note that this result implies equation \eqref{eq:phas_tran_2} in Theorem \ref{theo:phas_tran} however the proof presented on Section \ref{sec:phas_tran} is much simpler.
The values of $p_2$, $p_3 > p_c(\mathbb{Z}^2)$ will remain fixed.

Before proving Theorem \ref{theo:pha_tra_2}, let us introduce some concepts that will be useful in the proof.
We say that a path $\gamma = \{v_0, v_1, \ldots \}$ in $\mathbb{Z}^d$ is directed if its increments $v_{i+1} - v_{i}$ belong to the set of the canonical unit vectors $e_1,\ldots, e_d$.
A straight segment of length $k$ of the path $\gamma$ is any sequence $v_j,\ldots, v_{j+k}$ such that $|v_j - v_{j+l}| = l e_i$ for all $l \in 0,\dots, k$ and for some $i \in 1,\ldots, d$.
A path is said to be $2$-directed if it is directed and any of its straight segments has length at most $2$.
Let
\[
\tilde{p}_c(\mathbb{Z}^d) := \inf \{p \in [0,1]: \mathbb{P}_p (\text{$\exists$ an infinite $2$-directed path starting from $\mathbf{0}$})>0\}.
\]
By standard arguments based on directed percolation, one can show that this critical parameter for $2$-directed percolation satisfies $0 < \tilde{p}_c(\mathbb{Z}^2,2) < 1$.

Given a family of $2$-directed paths in $\mathbb{Z}^2$ starting at the origin, one can define the lowest path generated by this family in the usual way.
It is the the path $\gamma = \{v_0, v_1, \ldots\}$ such that each $v_i$ has the minimum vertical coordinate among all the sites that have the same horizontal coordinate as $v_i$ and that belong to at least one path in the family.

Recall that $R_1 (n; j,l)$, $i, j \in \mathbb{Z}$ are squares of side-length $n$ contained in $\mathcal{P}_1$ and define
\[
\mathcal{P}^{(n)}_1 := \{R_1 (n;j,l) ; ~j \in \mathbb{Z} , l \in \mathbb{Z}\}.
\]
which can be naturally regarded as a graph isomorphic to the $\mathbb{Z}^2$-lattice by adding an edge between $R_1(n;j,l)$ and $R_1(n;j',l')$ whenever $|j'-j| + |l'-l| = 1$.
The notion of a $2$-oriented path extends to $\mathcal{P}^{(n)}_1$ in the natural way.
Let, for the moment, $\gamma_n \subset \mathcal{P}^{(n)}_1$ be a fixed $2$-directed path of rectangles and consider
$\tilde{\gamma}_n = \left\{\tilde{\Gamma}_n (j,l,h) \in \Lambda_n ;\; \pi_1 (\tilde{\Gamma}_n(j,l,h)) \in \gamma_n \right\}$.
The set $\tilde{\gamma}_n$ will be called the lift of the path $\gamma_n$ and when considered as a sub-lattice of $\Lambda_n$ it is isomorphic to the $\mathbb{Z}^2$-lattice.

We now define the following event in $\{0,1\}^{\mathcal{P}_2} \times \{0,1\}^{\mathcal{P}_3} \times \{0,1\}^{\mathbb{Z}^3}$:
\[
\tilde{\mathcal{G}}(\gamma_n) = \left\{\text{$\exists$ an infinite path of good blocks in $\tilde{\gamma}_n$ starting from $\tilde{\Gamma}_n (0,0,0)$}\right\}.
\]

The following lemma shows that we can choose $n$ large enough so that the probability of finding infinite paths of good boxes in $\tilde{\gamma}_n$ is positive.
\begin{lemma}
\label{lemma:pha_tra_2_1}
There exists $n = n(p_2, p_3)$ and $p_{\bullet} = p_{\bullet}(p_2,p_3)$ such that, for every $p'' \geq p_{\bullet}$ and every $2$-oriented path $\tilde{\gamma}_n$ in $\mathcal{P}_1^{(n)}$,
\begin{equation}
\label{eq:pha_tran_2_1}
\mathbb{P}_{p_2} \times \mathbb{P}_{p_3} \times \mathbb{P}_{p''} \left( \tilde{\mathcal{G}}(\gamma_n)  \right) > 0.
\end{equation}
\end{lemma}
\begin{proof}
Consider the set of indices $I = \left\{ (j,l) \in \mathbb{Z}^2;~ R_1 (n,n;j,l) \in \gamma_n \right\}$ and 
$\tilde{I} = \{(j,l,h) \in \mathbb{Z}^3; ~ \tilde{\gamma}_n(j,l,h) \in \Gamma_n \}$.
Then $I$ can be regarded as a $2$-directed path in $\mathbb{Z}^2$ and $\tilde{I}$ as its lift in $\mathbb{Z}^3$. 
Let
\[
X(j,l,k) = \mathbf{1}_{\left\{\tilde{\Gamma}_n(j,l,h)~\text{is good}\right\}}
\]
and let $\mu$ denote its law on $\{0,1\}^{\tilde{I}}$.
The event $ \left\{ \tilde{\Gamma}_n(j,l,h)~ \text{is good} \right\}$ only depends on the $\omega_2$ and $\omega_3$ processes restricted to the projections $\pi_2 \left(\tilde{\Gamma}_n(j',l',h')\right)$ and $\pi_3 \left(\tilde{\Gamma}_n(j',l',h')\right)$ of rectangles $\tilde{\Gamma}_n(j',l',h')$ satisfying $|j'-j| + |l'-l| +|h'-h| \leq 1$.
Moreover, since the path $I$ is $2$-directed, it follows that there is a positive integer $M$ large enough such that the projections of each rectangle $\tilde{\Gamma}_n(j,l,h)$ into $\mathcal{P}_2$ and $\mathcal{P}_3$ overlap with at most the projection of $M$ other rectangles in $\Lambda_n$.
Then it follows that under $\mathbb{P}_{p_2} \times \mathbb{P}_{p_3} \times \mathbb{P}_{p''}$, $\left\{ X(j,l,k) \right\}_{(j,l,k) \in \tilde{I}}$ is a $M$-dependent percolation process on $\tilde{I}$.
Applying once more \cite[Theorem 0.0]{liggett97} we can find $u \in (0,1)$ large enough such that if 
\begin{equation}
\label{eq:Pp2Pp3P''X}
\mathbb{P}_{p_2} \times \mathbb{P}_{p_3} \times \mathbb{P}_{p''} ( X(j,l,k) =1) > u
\end{equation}
then \eqref{eq:pha_tran_2_1} holds.
The same arguments as in \ref{lemma:pat_blo_goo} can be applied in order to find $n = n(p_2, p_3)$ and $p_{\bullet} = p_{\bullet}(p_2, p_3)$ such that \eqref{eq:Pp2Pp3P''X} holds which concludes the proof.
\end{proof}

We are in the position to prove Theorem \ref{theo:pha_tra_2}.
Before we present the proof, let us just discuss the strategy.
We start by fixing $n$ and $p_{\bullet}$ as in the statement of the previous lemma.
Then choosing $p_1$ high enough then, under $\mathbb{P}_{p_1}$, the event that are infinite $2$-directed paths of $\omega_1$-open blocks in $\mathcal{P}_1^{(n)}$ have positive probability. 
Now, conditioning on the the realization of such a path, and assuming that $p'' \geq p_{\bullet}$ the previous lemma garantees that the probability of finding an infinite path of good blocks in its lift is positive.
The proof is concluded by using Lemma \ref{lemma:geo_2} in order to relate such a path of good blocks to an infinite path of $\omega$-open sites.

\begin{proof}[Proof of Theorem \ref{theo:pha_tra_2}]
Let us fix $\varepsilon = \min\{1-p_c(\mathbb{Z}^2,2)^{1/n^2}, 1-p_{\bullet}\}$ where $n=n(p_2,p_3)$ and $p_{\bullet} = p_{\bullet}(p_2,p_3)$ are given as in Lemma \ref{lemma:pha_tra_2_1}.

Given $\omega_1 \in \{0,1\}^{\mathcal{P}_1}$ we say that the square $R_1(n;j,l) \subset \mathcal{P}_1$ is open if $\omega_1 (v) =1$ for all $v \in R_1(n;j,l)$.
Let 
\[
\mathcal{A}_1 = \left\{ \exists ~\text{an infinite $2$-directed path of open rectangles in $\mathcal{P}_1^{(n)}$ starting at $R_1(n,n;0,0)$}\right\}
\]
Then
\begin{equation}
\label{eq:pp1}
\mathbb{P}_{p_1} \left( \omega_1 \in {\mathcal{A}_1} \right) = \mathbb{P}_{p^{n^2}_1} \left( \left\{ \text{$\exists$ an infinite $2$-directed path starting at $\mathbf{0}$} \right\} \right) >0,
\end{equation}
where the inequality follows from the definition of $\varepsilon$.

For each $\omega_1 \in \mathcal{A}_1$ let $\gamma_n$ be the lowest $2$-directed path of open rectangles in $\mathcal{P}_1^{(n)}$ starting at $R_1(n;0,0)$ and let $\tilde{\gamma}_n$ denote its lift in $\Lambda_n$.

Let $\{ \tilde{\Gamma}_n(0,0,0) \leftrightarrow \infty \}$ denote the event that there is an infinite path of $\omega$-open sites $v$ in $\mathbb{Z}^3$ with the starting point belonging to the block $\tilde{\Gamma}_n(0,0,0)$.
Writing $\mathbb{P}$ for $\mathbb{P}_{p_1} \times \cdots \times \mathbb{P}_{p_2}$ we have
\begin{equation}
\label{eq:pha_tra_2_3}
\begin{split}
& \mathbb{P} \left( \omega \in \{\tilde{\Gamma}_n(0,0,0) \leftrightarrow \infty \}\right) = 
 \mathbb{E} \left[ \mathbb{P} \left( \omega \in \left\{ \tilde{\Gamma}_n(0,0,0) \leftrightarrow 
\infty \right\} \Big{|} \mathcal{F}_1 \right) \right] \geq \\
& \mathbb{E}_{p_1} \left[ \mathbf{1}_{\{\omega_1 \in \mathcal{A}_1\}} \mathbb{P} \left( \omega \in \left\{ \tilde{\Gamma}_n(0,0,0) \leftrightarrow 
\infty ~\text{in}~ \tilde{\gamma}_n \right\} \Big{|} \mathcal{F}_1 \right)(\omega_1) \right].
\end{split}
\end{equation}

Now, on the event $\{\omega_1 \in \mathcal{A}_1\}$, all the sites in $\pi_1(\tilde{\gamma}_n)$ are $\omega_1$-open. So, in view of Lemma \ref{lemma:geo_2} and Remark \ref{rmk:geo_2}, in order for $\omega$ to belong to $\left\{ \tilde{\Gamma}_n(0,0,0) \leftrightarrow 
\infty ~\text{in}~ \tilde{\gamma}_n (\omega_1) \right\}$ it suffices that $(\omega_2, \omega_3, \omega'') \in \tilde{\mathcal{G}_n}(\tilde{\gamma}_n(\omega_1))$.

\begin{equation}
\label{eq:pha_tra_2_4}
\begin{split}
& \mathbb{E}_{p_1} \left[ \mathbf{1}_{\{\omega_1 \in \mathcal{A}_1\}} \mathbb{P} \left( \omega \in \left\{ \tilde{\Gamma}_n(0,0,0) \leftrightarrow 
\infty ~\text{in}~ \tilde{\gamma}_n \right\} \Big{|} \mathcal{F}_1 \right)(\omega_1) \right] \geq \\
& \mathbb{E}_{p_1} \left[ \mathbf{1}_{\{\omega_1 \in \mathcal{A}_1\}} \mathbb{P}_{p_2} \times \mathbb{P}_{p_3} \times \mathbb{P}_{p''} \left( \tilde{\mathcal{G}_n}(\tilde{\gamma}_n(\omega_1))\right) \right] >0,
\end{split}
\end{equation}
where the last inequality follows from Lemma \ref{lemma:pha_tra_2_1} in view of the fact that $\varepsilon \leq 1-p_{\bullet}$.

Plugging \eqref{eq:pha_tra_2_4} into \eqref{eq:pha_tra_2_3} we obtain $\mathbb{P} (\omega \in \{ \tilde{\Gamma}_n(0,0,0) \leftrightarrow \infty \} ) >0$.

Consider for each $i \in \{1,2,3\}$, the increasing events
$\mathcal{B}_i := \left\{ \pi_i \left( \tilde{\Gamma}_n (0,0,0) \right) \text{is $\omega_i$-open} \right\}$.
If $\omega_i \in \mathcal{B}_i$ for all $i = 1,2,3$ then all sites in $\tilde{\Gamma}_n (0,0,0)$ are $\omega'$-open.
Then $\{ \mathbf{0} \leftrightarrow \infty \}$ is contained in $\left\{\omega \in \{ \tilde{\Gamma}_n(0,0,0) \leftrightarrow \infty \} \right\} \cap \{\omega_i \in\mathcal{B}_i, \text{ for $i=1,2,3$}.\}$ 
Using the Harris-FKG inequality and the last lemma we have that:
\begin{equation*}
\begin{split}
& \mathbb{P}_{\mathbf{p}} \left( \left\{ \mathbf{0} \leftrightarrow \infty \right\} \right) \ge  \mathbb{P} \left(\omega \in \left\{ \tilde{\Gamma}_n(0,0,0) \leftrightarrow \infty \right\} \right) \prod_{i=1,2,3} \mathbb{P}_{p_i}(\{\mathcal{B}_i\}) > 0.
\end{split}
\end{equation*}
\end{proof}

\begin{remark}
\label{r:p_*}
The same ideas used in the proof of Theorem \ref{theo:pha_tra_2} can be employed to obtain an upper bound for the critical value $p_*$ introduced in \eqref{e:p_*}.
In fact, if $p_1 > \tilde{p}_c(\mathbb{Z}^2)^{1/3}$, then we can find, with positive probability a 2-directed open path $\gamma_1 \subset \mathcal{P}_1$.
Conditioned in $\gamma_1$ the restriction of the process $\omega$ to $\tilde{\gamma}_1 := \pi_1^{-1} (\gamma_1)$ is a one-dependent percolation process.
Thinking of $\tilde{\gamma}_1$ as a copy of $\mathbb{Z}^2$, one can again use a simple block argument to show that if $p_2$ and $p_3$ are bigger than $\tilde{p}_c({\mathbb{Z}^2,2})^{1/3}$, then one can find, with probability one, an infinite path of $(\omega_2 \cdot \omega_3)$-open sites in $\tilde{\gamma_1}$.
This shows that $p_* \leq \tilde{p}_c(\mathbb{Z}^2,2)^{1/3}$.
\end{remark}

\section{The number of infinite connected components}
\label{number_clusters}

Recall that, $N(\omega)$ denotes the number of infinite connected components of $\omega \in \Omega$. 
In this section we prove Theorem \ref{theo:num_clu} establishing that, $N$ is either $0$, $1$ or $\infty$  almost surely.

We start with a more general setting.
Let $\mu$ be a translation invariant probability measure on $\mathcal{F}$.
In \cite{Newman81_1} it was proved that, if $\mu$ is ergodic and satisfies the so-called \emph{finite energy condition} then $N \in \{0, 1, \infty\}$.
Roughly speaking, satisfying the finite energy condition means that the image of events of positive probability by local modifications remains an event of positive probability (see \cite{Newman81_1} for the precise definition).

For proving this result, the authors in \cite{Newman81_1} start assuming that $N$ is finite and strictly greater then one.
Then they fix a box large enough so that the probability that it intersects all the $N$ infinite connected components is positive.
Performing the local modification that consists in opening all sites in that box while keeping the configuration out of it untouched they merge all the connected components in a unique one. 
By the finite energy condition, one concludes that the probability of having a unique connected component is positive.
However this contradicts the fact that $N$ is a constant strictly greater than one, almost surely.

However, the measure $\mathbb{P}_{\mathbf{p}}$ does not satisfy the finite energy condition.
For instance, on the event that all neighbors of the origin are open, the origin itself is open with probability one.
In particular flipping the state of the origin to $0$ is a local modification that leads to an event of probability zero.

Moreover, the proof presented in \cite{Newman81_1} cannot be directly adapted to $\mathbb{P}_{\mathbf{p}}$.
In fact, in order to open all the sites in a given box, one would need to modify the state of vertices lying all along the lines that intersect this box and that are parallel to the coordinate axis.
Therefore, as one attempts to merge all the infinite connected components into a single one, it could be that case that other infinite components would appear elsewhere.
Then, in principle, this argument does not lead to any contradiction.
In order to prove Theorem \ref{theo:num_clu} we overcome this difficulty by dealing with the density of the connected components rather than their number, as we show below.

We say that a subset $A \subset \mathbb{Z}^d$ has density $\rho$ if for any increasing sequence of rectangles $R_1 \subset R_2 \subset \ldots$ with $\cup_{i \geq 1} R_i = \mathbb{Z}^d$ the limit
\begin{equation*}
\lim_{i \to \infty} \frac{|A \cap R_i|}{|R_i|}
\end{equation*}
exists and equals $\rho$.
It has been proved by Burton and Keane \cite[Theorem 1]{Burton89} that, for any translation invariant probability $\mu$ on $\mathcal{F}$, all the connected components have a density, $\mu$-almost surely. 

For $v$ and $w \in\mathbb{Z}^d$, let $T_v(w) = w+v$. 
Assuming that $\mu$ is translation invariant and ergodic with respect to the transformation $T_v$ for a $v \in \mathbb{Z}^d$, then $N$ is constant, $\mu$-almost surely.
Assume further that $N \geq 1$ $\mu$-almost surely and define a ranked density vector \emph{i.e.,} a random vector $\rho = \rho(\omega)$ whose entries are the densities of the infinite connected components of $\omega$ arranged in a non-increasing fashion.
More specifically define:
\begin{equation}
\label{eq:ran_den}
 \rho := \left\{
 \begin{array}{ll}
   ( \rho_1, \ldots , \rho_N ),  & \text{if}~ N < \infty, \\
   ( \rho_1, \rho_2, \ldots ), & \text{if}~ N = \infty;
 \end{array}
 \right.
\end{equation}
where $\rho_1 (\omega) \ge \rho_2 (\omega) \ge \cdots$ are the densities of the infinite connected components.
Since $\rho$ is invariant under $T_v$ we have that $\rho$ is almost surely constant.

The next proposition establishes that, under the assumption that $N<\infty$, no infinite connected components has zero density.
\begin{proposition}
\label{prop:den_fin_clu}
Let $\mu$ be a translation-invariant ergodic probability measure on $\Omega$ for which $0< N < \infty$ almost surely and let $\rho$ be the ranked density vector given by (\ref{eq:ran_den}).
Then all entries of $\rho$ are strictly positive constants, $\mu$-almost surely.
\end{proposition}
\begin{proof}
As mentioned before, by ergodicity, each entry of $\rho$ is constant.
Assume, in order to find a contradiction, that there is an index $k \in \{1,\ldots, N\}$ for which $\rho_k=0$.
We can assume also that $k$ is the smallest index satisfying this property.
Note that, by the definition of the vector $\rho$, we have $\rho_j = 0$ for all $k \leq j \leq N$.

Let, for each $j \in \{1, \ldots, N\}$, $C_j$ stand for the connected component corresponding to the $j$-th entry of $\rho$ (selected in a arbitrary order when there are ties).
Define $C' = \cup_{j=k}^{N} C_j$.
Then $C'$ is a non-empty random infinite subset of $\mathbb{Z}^d$ whose distribution is invariant under lattice translations.
In particular, $\mu (\{\mathbf{0} \in C'\}) = \mu (\{v \in C' \})$ for all $v \in \mathbb{Z}^d$.

Let $R_1 \subset R_2 \subset \cdots$ be any increasing sequence of rectangles such that $\cup_{j=1}^{\infty} R_j = \mathbb{Z}^d$.
Then, since $C_k, \ldots, C_N$ have density $0 = \rho_k = \cdots = \rho_N$ we have that:
\begin{equation*}
 \lim_{n \to \infty} \frac{1}{|R_n|} \sum_{v \in R_n} \mathbf{1}_{\{v \in C'\}} = (N - k + 1)\rho_k = 0.
\end{equation*}
Integrating the left-hand side with respect to $\mu$, using the Bounded Convergence Theorem and translation-invariance we have that $\mu(\{\mathbf{0} \in C'\}) = 0$, so that, $\mu(\{C'= \emptyset\}) = 1$.
The proof is finished since this contradicts the fact that $C'$ is non-empty, almost surely.
\end{proof}

From now on we fix $\mathbb{P}_{\mathbf{p}}$ with $0< p_i < 1$ for all $i$.
For this measure, it is not the case that any $T_v$ is ergodic.
For instance, if we take $v = (1,0,\ldots, 0)$ then the event $\{\omega(kv) = 0 ~\text{for all}~ k \in \mathbb{Z}\}$ is invariant under $T_v$, however it has probability equal to $1 - p_1 \notin \{0,1\}$.

On the other hand, when $v$ has at least $d-1$ nonzero coordinates then it is the case that $T_v$ is mixing.
In order to see that, let $\mathcal{A}$ and $\mathcal{B}$ be two cylinders in $\mathcal{F}$ whose occurrences are determined by the states of the sites in the finite sets $A \subset \mathbb{Z}^d$ and $B\subset \mathbb{Z}^d$ respectively.
Then there exists a $n_0$ (depending on $A$ and $B$) such that $\pi_i(T^n_v B) \cap \pi_i(A) = \emptyset $ for all $i =  1,\ldots, d$ and for all $n> n_0$.
So, for all such indices $n > n_0$, we have that $\{\omega(w);~ w \in \mathcal{A}\}$ and $\{\omega(w); ~w \in T_v^n \mathcal{B} \}$ are independent sets of random variables, which implies the mixing condition.
In particular, $\mathbb{P}_{\mathbf{p}}$ is ergodic with respect to $T_e$, where $e = (1,1, \ldots, 1)$.
Since, $T_e N = N$ for all $\omega \in \Omega$, then $N$ is a random variable invariant with respect to $T_e$ and then $N$ is constant almost surely.

\begin{proof}[Proof of Theorem \ref{theo:num_clu}:]
Assume that $1 < N < \infty$ and let $\rho = (\rho_1, \ldots, \rho_N)$ be the ranked density vector defined in (\ref{eq:ran_den}).
Denote by $C(1), \ldots, C(N)$ be the infinite connected components corresponding to each of the entries of $\rho$.
When $\rho$ has entries with the same value, pick the corresponding  connected components arbitrarily among the possible choices.

From proposition \ref{prop:den_fin_clu} we have that all the entries of $\rho$ are strictly positive.
Thus we can fix a positive $n_0$ such that for all $n > n_0$ the probability that $B^d(n)$ intersects all the infinite connected components $C(1), \ldots, C(N)$ is positive.
So, fixing $n > n_0$ and denoting by $\mathcal{A}$ the event in $\Omega_1 \times \cdots \times \Omega_d$ given by
\begin{equation*}
\mathcal{A} = \left\{ \omega \in \{B^d(n) \text{ intersects all the connected components } C(1), \ldots, C(N) \} \right\}
\end{equation*}
we have that $\mathbb{P}_{p_1} \times \cdots \times \mathbb{P}_{p_d} (\mathcal{A}) > 0$.

Define the mapping $\phi : \Omega_1 \times \cdots \times \Omega_d \to \Omega_1 \times \cdots \times \Omega_d$ by setting
\begin{equation*}
\phi(\omega_1, \ldots, \omega_d) = (\phi_1(\omega_1),\cdots, \phi_d(\omega_d)),
\end{equation*}
where,
\begin{equation*}
 \phi_i (\omega_i) (v) = \left\{
 \begin{array}{cl}
   1, & \text{if}~ v \in B^{d-1}_i (n) \\
   \omega_i(v), & \text{ if } v \notin B^{d-1}_i (n).
\end{array}
\right.
\end{equation*}

On the event  $\phi(\mathcal{A})$ the vector $\rho$ has an entry with value at least equal to $\rho_1 + \cdots + \rho_N > \rho_1$, since opening the sites in $B^{d} (n)$ merges the connected components $C(1), \ldots, C(N)$ in a single $\omega$-connected component.
Since $\mathcal{A}$ is increasing then the fact that $\mathbb{P}_{p_1} \times \cdots \times \mathbb{P}_{p_d} (\mathcal{A}) > 0$, implies that  $\mathbb{P}_{p_1} \times \cdots \times \mathbb{P}_{p_d} (\phi(\mathcal{A})) > 0$.
This implies that the vector $\rho$ has an entry at least equal to $\rho_1 + \cdots + \rho_N$ with positive probability 
contradicting the fact that the first entry of $\rho$ is constant and equal to $\rho_1$ almost surely.
\end{proof}

\section{Appendix}

\subsection{Proof of Lemma \ref{lemma:geo}}
In this section we give a prove of lemma \ref{lemma:geo}.

\begin{proof}[Proof of Lemma \ref{lemma:geo}:]
Let $h = |h(\gamma)| = |h(\gamma')|$.
We will use induction in $h$.
We also restrict ourselves to the case $h(v_0) = h(w_0) = 0$ and $h(v_m) = h(w_{m'}) = h > 0$.
The proof of any other case is similar.

We first consider $h = 1$.
Let $\gamma_H = \{ v_0, \ldots, v_{m-1} \}$ and $\gamma'_H = \{ w_0, \ldots, w_{m'-1} \}$ the horizontal parts of the paths $\gamma$ and 
$\gamma'$ (note that they can be a single point if $n=1$ or $m'=1$).
We can thus define:
\begin{equation}
 \label{eq:gamm_h_w_0}
 \gamma_H \times w_0 = \{ v_0 \times w_0, v_1 \times w_0, \ldots, v_{m-1} \times w_0\} \text{ and }
\end{equation}
\begin{equation}
 \label{eq:v_m-1_gamma_prim_h}
 v_{m-1} \times \gamma'_H = \{v_{m-1} \times w_0, v_{m-1} \times w_1 , \ldots, v_{m-1} \times w_{m'-1}\}. 
\end{equation}
The paths above are well defined, since $h(v_j) = h(w_i) = 0$ for any $v_j$ and $w_j$ appearing at the right hand side of those equations.
Since the ending point of $\gamma_H \times w_0$ is equal to the starting point of $v_{m-1} \times \gamma'_H$ we can define
\begin{equation}
 \gamma_H \times \gamma'_H = (\gamma_h \times w_0) * (v_{m-1} \times \gamma'_H). 
\end{equation}
It is then straightforward to check that $\pi_2 (\gamma_H \times \gamma'_H) = \gamma_H$ and that $\pi_3(\gamma_H \times \gamma'_H) = \gamma'_H$ and that $\gamma_H \times \gamma'_H$ starts at $v_0 \times w_0$ and ends at $v_{m-1} \times w_{m'-1}$.
If we now let $\gamma_V = \{v_{m-1}, v_m\}$ and $\gamma'_V = \{w_{m'-1}, w_{m'}\}$ be the vertical parts of $\gamma$ and $\gamma'$ respectively, and define
\begin{equation}
 \gamma_V \times \gamma'_V = \{ v_{m-1} \times w_{m'-1}, v_m \times w_{m'}\}
\end{equation}
 then $\pi_2(\gamma_V \times \gamma'_V) = \gamma_V$ and $\pi_3(\gamma_V \times \gamma'_V) = \gamma'_V$.
Finally let us set
\begin{equation}
 \gamma \times \gamma' = (\gamma_H \times \gamma'_H)*(\gamma_V \times \gamma'_V)
\end{equation}
which is a path starting at $v_0 \times w_0$, ending at $v_m \times w_{m'}$ and satisfying \eqref{eq:proj_gamm_gamm_prim}.
This finishes the proof for $h=1$.

Now let us consider the case $h = h_0 +1$ where $h_0 \geq 1$ is fixed.
Assuming that the lemma holds for any pair of compatible paths having height no greater than $n_0 \geq 1$ we are going to show that the lemma holds for $\gamma$ and $\gamma'$ finishing thus the proof.

We begin by splitting the paths $\gamma$ and $\gamma'$ into several up and down-excursions having variation $n_0$.
For that let $t_0 = t'_0 = 0$ and define inductively for all $n \geq 1$:

\begin{equation}
 \begin{split}
  & t_{2n-1} = \inf \{ j > t_{2n-2}; h(v_j) = h_0\}\\
  & t_{2n} = \inf \{ j > t_{2n-1}; h(v_j) = 0 \},
 \end{split}
\end{equation}
with the convention that $\inf \emptyset = \infty$.
Let $f$ be defined so that $2f-1$ is the number of finite elements in the sequence $t_0, t_1, t_2, \ldots$.
Then $f$ represents the number of excursions from height zero to height $h_0$.
Similarly we define $t'_{2n-1}$, $t'_{2n}$ and $f'$, the analogous indices for the path $\gamma'$.
We will assume that $f> 1$ and $f'>1$.
The other cases are simpler to deal with.

Then we have the following sequences:
\begin{equation}
 t_0 < t_1 < \cdots < t_{2f-1} \text{  and  } t'_0 < t'_1 < \cdots < t'_{2f'-1}
\end{equation}
and the paths:
\begin{equation}
 \begin{split}
  & \gamma_j = \{ v_{t_j}, \ldots, v_{t_{j+1}}\} \text{ for } j=0,\ldots, 2f-2 \\
  & \gamma'_j = \{ w_{t'_j}, \ldots, w_{t'_{j+1}}\} \text{ for } j=0, \ldots, 2f'-2.
 \end{split}
\end{equation}
Note that, if $j$ is even, then $\gamma_j$ and $\gamma'_j$ are paths with variation equal to $h_0$ with the starting point having height equal to zero and the ending point having height equal to $h_0$.

Let us also define the following paths:
\begin{equation}
\begin{split}
 & \eta = \overline{\left(\overline{\gamma_0} \wedge 0\right)} \text{ and }\\
 & \zeta = \overline{\left(\overline{\gamma_{2f-2}} \wedge 0\right)}.
\end{split}
\end{equation}
We also define the paths $\eta'$ and $\zeta'$ as the analogues of $\eta$ and $\zeta$ for the path $\gamma'$.
In words, $\eta$ can be described as the set of sites that would be traversed when one travels along $\gamma_0$ after visiting height zero for the last time.
Note that $\eta$ connects a site lying at height zero to a site lying at height $h_0$ without ever touching these two heights in between.
The paths $\zeta$, $\eta'$ and $\zeta'$ can be described in a similar fashion.

Having already defined the paths $\gamma_0, \ldots, \gamma_{2f-2}$ and $\gamma'_0, \ldots, \gamma'_{2f'-2}$ let us now define:
\begin{equation}
\begin{split}
 & \gamma_{2f-1} = \overline{\zeta} \wedge 1 ~~~~~\text{ and }~~~ \gamma_{2f} = \gamma \setminus (\gamma_0 * \cdots * \gamma_{2f-1}) \\
 & \gamma'_{2f'-1} = \overline{\zeta'} \wedge 1 ~~~\text{ and }~~~ \gamma_{2f'} = \gamma \setminus (\gamma'_0 * \cdots * \gamma_{2f'-1}).
\end{split}
\end{equation}
Thus we can write
\begin{equation}
\begin{split}
 &\gamma = \gamma_0 * \gamma_1 * \cdots * \gamma_{2f-2} * \gamma_{2f-1} * \gamma_{2f} ~~~~\text{ and } \\
 &\gamma' = \gamma'_0 * \gamma'_1 * \cdots * \gamma_{2f'-2} * \gamma_{2f'-1} * \gamma_{2f'}.
\end{split}
\end{equation}

Roughly speaking, this equation express the decomposition of $\gamma$ (and similarly for $\gamma'$) as follows:
Go up along $\gamma$ until hitting height $h_0$. 
Then go down back to height zero. 
Repeat it for $f-1$ times and then go up again until hitting height $h_0$.
At that point the path $\overline{\gamma}_{2f-2}$ has just been traversed from bottom to top.
Now go along its reversal $\overline{\gamma}_{2f-2}$ stopping at the step just after reaching height zero.
This corresponds to $\gamma_{2f-1}$.
Then follow $\gamma$ from this point on until hitting its last site $v_m$.

Note that $\gamma_0$ and $\gamma'_0$ are two compatible paths of variation $h_0$ so, by the induction hypothesis there is a path
\begin{equation}
 \gamma_0 \times \gamma'_0
\end{equation}
starting at $v_0 \times w_0$ and ending at $v_{t_1} \times w_{t'_1}$ and such that $\pi_2(\gamma_0 \times \gamma'_0) = \gamma_0$ and $\pi_3(\gamma_0 \times \gamma'_0) = \gamma'_0$.
Also, for each $0< j<2f-2$ odd, $\gamma_j$ and $\overline{\eta'}$ are compatible paths of height $h_0$.
Similarly for each $0 < j \leq 2f-2$ even, $\gamma_j$ and $\eta'$ also constitute a pair of compatible paths.
So, we can pick the paths $\gamma_j \times \eta'$ for $j$ even and $\gamma_j \times \overline{\eta'}$ for $j$ odd.
All those paths have their projections into $\mathcal{P}_2$ equal to $\gamma_j$ and their projections into $\mathcal{P}_3$ equal to $\eta'$ or $\overline{\eta'}$.
Also the ending point of each one of them is the starting point of the following one.
So we can define:
\begin{equation}
 \gamma \times \eta' = (\gamma_1 \times \overline{\eta'}) * (\gamma_2 \times \eta') * \ldots * (\gamma_{2f-2} \times \eta') 
\end{equation}
and it follows that $\pi_2 (\gamma \times \eta') = \gamma$ and $\pi_3(\gamma \times \eta') = \eta'$.

Following an analogous procedure we can pick the paths
$\zeta \times \gamma'_j$ for $j$ even and $\overline{\zeta} \times \gamma'_j$ for $j$ odd ($1 \leq j \leq 2f'-2$) and then define:
\begin{equation}
 \zeta \times \gamma' = (\overline{\zeta} \times \gamma'_1) * (\zeta \times \gamma'_2) * \cdots * (\zeta \times \gamma'_{2f'-2}).
\end{equation}
Note that this path starts at $v_{t_{2f-2}} \times w_{t'_1}$ and ends at $v_{t_{2f-2}} \times w_{t'_{2f'-2}}$.
Also they satisfy that $\pi_2(\zeta \times \gamma') = \zeta$ or $\overline{\zeta}$ and $\pi_3 (\zeta \times \gamma') = \gamma'$.

Noting also that $\gamma_{2f-1}$ and $\gamma'_{2f'-1}$ are compatible paths with variation equal to $h_0 - 1$ starting at $v_{t_{2f-2}}$ and $w'_{t_{2f'-2}}$ respectively and that $\gamma_{2f}$ and $\gamma'_{2f'}$ are compatible paths of variation $h_0$, we can then pick
\begin{equation}
 \gamma_{2f-1} \times \gamma_{2f'-1} \text{  and  } \gamma_{2f} \times \gamma'_{2f'}
\end{equation}
and concatenate then in order to have a path starting at $v_{t_{2f-2}} \times w_{t'_{2f'-2}}$ and finishing at $v_m \times w_{m'}$ and having:
\begin{equation}
\begin{split}
 &  \pi_2 \left( (\gamma_{2f-1} \times \gamma'_{2f'-1}) * (\gamma_{2f} \times \gamma_{2f'}) \right) \subset \gamma_{2f-1} \cup \gamma_{2f} \text{ and }\\
 & \pi_3 \left( (\gamma_{2f-1} \times \gamma'_{2f'-1}) * (\gamma_{2f}*\gamma'_{2f'}) \right) \subset \gamma'_{2f'-1} \cup \gamma'_{2f'}.
\end{split}
\end{equation}

Finally let us define:
\begin{equation}
 \gamma \times \gamma' = (\gamma_0 \times \gamma'_0) * (\gamma \times \eta') * (\zeta \times \gamma') * (\gamma_{2f-1} \times \gamma'_{2f'-1}) * (\gamma_{2f} \times \gamma_{2f'}).
\end{equation}
which is a path in $\mathbb{Z}^3$ starting at $v_0 \times w_0$, ending at $v_m \times w_{m'}$ and satisfying \eqref{eq:proj_gamm_gamm_prim}.
\end{proof}

\subsection{Proof of equation \ref{eq:cro_log_2}}
\label{app:cro_log}
A $*$-path in $\mathbb{Z}^2$ is a sequence $\{v_0, v_1, \ldots, v_r \}$ of sites such that $|v_j - v_{j-1}|_{\infty} =1$ for all $j = 1, \ldots , r$ (where $| \cdot |_{\infty} $ stands for the $l_{\infty}$-distance in $\mathbb{Z}^2$).
Denote by $\mathcal{A}^{*}(n,m;k,l)$ the event that there exists a $*$-path of closed sites crossing $R(n,m;k,l)$ from its left to its right.
Then, by duality,
\begin{equation}
 \label{eq:dual}
\mathcal{B}(n,m;k,l) \text{ occurs  if and only if } \mathcal{A}^{*}(n,m;k,l) \text{ does not occurs}.
\end{equation}

If now $\mathbb{Z}^2_{*}$ stands for the graph with vertex set $\mathbb{Z}^2$ and with an edge between each pair of vertices lying at $l_{\infty}$-distance $1$ from each other and $p_c(\mathbb{Z}_{*}^2)$ the critical density for Bernoulli site percolation on this lattice, then we have that $p_c(\mathbb{Z}^2) + p_c(\mathbb{Z}^2_{*}) = 1$ (see \cite{Russo81} for a proof).
Thus $p>p_c(\mathbb{Z}^2)$ implies that $1-p < p_c(\mathbb{Z}^2_{*})$ so by the analogue of equation \eqref{eq:mens} for Bernoulli site percolation on $\mathbb{Z}^2_{*}$ (see \cite{Menshikov86}), there exists $\psi_{*} = \psi_{*}(p) > 0$ such that:
\begin{equation}
\label{eq:dec_*}
\mathbb{P}_p \left( \left\{ \text{there is a}~ *\text{-path of closed sites from}~ \mathbf{0}~ \text{to}~ \partial{B(n)} \right\} \right) \leq e^{-\psi_{*}(p)n}.
\end{equation}

For any constant $c > 0$ we have
\[
\mathbb{P}_p \left( \mathcal{A}^{*} \left( \lceil c\log n \rceil, n \right) \right) \leq n\;\mathbb{P}_p \left( \left\{ \text{there is a}~ *\text{-path of closed sites from}~ \mathbf{0}~ \text{to}~ \partial{B(\lceil c\log n \rceil)} \right\} \right)
\]

Then using \eqref{eq:dual} and equation \eqref{eq:dec_*} we have that:
\begin{equation}
 \label{eq:cro_log}
\begin{split}
 \mathbb{P}_p \left( \mathcal{B} \left( \lceil c\log n \rceil, n \right) \right) & = 1 - \mathbb{P}_p \left( \mathcal{A}^{*} \left( \lceil c\log n \rceil, n \right) \right) \\
& \geq 1 - ne^{-\psi_{*}(p)c \log{n}} = 1- n^{1- c\psi_{*}(p)}
\end{split}
\end{equation}
Now if we fix $c > \psi_{*}(p)^{-1}$ equation (\ref{eq:cro_log}) yields,
\begin{equation}
\lim_{n \to \infty} \mathbb{P}_p \left( \mathcal{B} \left( \lceil c\log n \rceil, n \right) \right) = 1.
\end{equation}

\section*{Acknowledgements}
The authors thank {N.A.M. Ara\'ujo}, {H.J. Herrmann}, {K.J. Schrenk} {A.-S. Sznitman} and {A. Teixeira} for  discussions,  and {V. Tassion} for suggesting the name "Bernoulli line percolation" for this model.
M.R.H. was partially supported by CNPq grants 140532/2007-2, 248718/2013-4. The research of V.S. was supported in part by Brazilian CNPq grants 308787/2011-0 and 476756/2012-0 and FAPERJ grant E-26/102.878/2012-BBP.
This work was also supported by ERC AG ``COMPASP" and ESF RGLIS grants.

\bibliographystyle{alpha}
\bibliography{tese}

\newcommand{\etalchar}[1]{$^{#1}$}
\begin{thebibliography}{DCT15}

\bibitem[AB87]{Aizenman87}
M.~Aizenman and D.~J. Barsky.
\newblock Sharpness of the phase transition in percolation models.
\newblock {\em Comm. Math. Phys.}, 108(3):489--526, 1987.

\bibitem[BBS00]{Balister00}
P.~N. Balister, B.~Bollob{\'a}s, and A.~M. Stacey.
\newblock Dependent percolation in two dimensions.
\newblock {\em Probab. Theory Related Fields}, 117(4):495--513, 2000.

\bibitem[BK89]{Burton89}
R.~M. Burton and M.~Keane.
\newblock Density and uniqueness in percolation.
\newblock {\em Comm. Math. Phys.}, 121(3):501--505, 1989.

\bibitem[BSS14]{Basu14}
Riddhipratim Basu, Allan Sly, and Vladas Sidoravicius.
\newblock Scheduling of non-colliding random walks.
\newblock {\em ArXiv Mathematics e-prints}, 2014.

\bibitem[CCN87]{Chayes87}
J.~T. Chayes, L.~Chayes, and C.~M. Newman.
\newblock Bernoulli percolation above threshold: an invasion percolation
  analysis.
\newblock {\em Ann. Probab.}, 15(4):1272--1287, 1987.

\bibitem[DCT15]{Duminil-Copin15}
Hugo Duminil-Copin and Vincent Tassion.
\newblock A new proof of the sharpness of the phase transition for bernoulli
  percolation and the ising model.
\newblock {\em ArXiv Mathematics e-prints}, 2015.

\bibitem[G{\'a}c00]{Gacs00}
P.~G{\'a}cs.
\newblock The clairvoyant demon has a hard task.
\newblock {\em Combin. Probab. Comput.}, 9(5):421--424, 2000.

\bibitem[G{\'a}c11]{Gacs11}
P.~G{\'a}cs.
\newblock Clairvoyant scheduling of random walks.
\newblock {\em Random Structures Algorithms}, 39(4):413--485, 2011.

\bibitem[Gri99]{Grimmett99}
G.~Grimmett.
\newblock {\em Percolation}, volume 321 of {\em Grundlehren der Mathematischen
  Wissenschaften}.
\newblock Springer-Verlag, Berlin, second edition, 1999.

\bibitem[Har60]{Harris60}
T.~E. Harris.
\newblock A lower bound for the critical probability in a certain percolation
  process.
\newblock {\em Proc. Cambridge Philos. Soc.}, 56:13--20, 1960.

\bibitem[HST]{Hilario12}
M.~R. Hil\'ario, V.~Sidoravicius, and A.~Teixeira.
\newblock Cylinders' percolation on $\mathbb{R}^3$.
\newblock {\em Probab. Theory Related Fields -- to appear}.

\bibitem[Kan86]{Kantor86}
Yacov Kantor.
\newblock Three-dimensional percolation with removed lines of sites.
\newblock {\em Phys. Rev. B}, 33:3522--3525, Mar 1986.

\bibitem[LSS97]{liggett97}
T.~M. Liggett, R.~H. Schonmann, and A.~M. Stacey.
\newblock Domination by product measures.
\newblock {\em Ann. Probab.}, 25(1):71--95, 1997.

\bibitem[Men86]{Menshikov86}
M.~V. Menshikov.
\newblock Coincidence of critical points in percolation problems.
\newblock {\em Dokl. Akad. Nauk SSSR}, 288(6):1308--1311, 1986.

\bibitem[NS81a]{Newman81_1}
C.~M. Newman and L.~S. Schulman.
\newblock Infinite clusters in percolation models.
\newblock {\em J. Statist. Phys.}, 26(3):613--628, 1981.

\bibitem[NS81b]{Newman81_2}
C.~M. Newman and L.~S. Schulman.
\newblock Number and density of percolating clusters.
\newblock {\em J. Phys. A}, 14(7):1735--1743, 1981.

\bibitem[Pet08]{Pete08}
G.~Pete.
\newblock Corner percolation on {$\Bbb Z^2$} and the square root of 17.
\newblock {\em Ann. Probab.}, 36(5):1711--1747, 2008.

\bibitem[Rus81]{Russo81}
L.~Russo.
\newblock On the critical percolation probabilities.
\newblock {\em Z. Wahrsch. Verw. Gebiete}, 56(2):229--237, 1981.

\bibitem[SHS{\etalchar{+}}]{Schrenk15}
K.J. Schrenk, M.R. Hil\'ario, V.~Sidoravicius, N.A.M. Ara\'ujo, H.J. Herrmann,
  M.~Thielmann, and A.~Teixeira.
\newblock How many random drills are needed to collapse a wood cube?
\newblock {\em in preparation}.

\bibitem[SS09]{Sidoravicius09}
V.~Sidoravicius and A.-S. Sznitman.
\newblock Percolation for the vacant set of random interlacements.
\newblock {\em Comm. Pure Appl. Math}, 62(6):831--858, 2009.

\bibitem[SS10]{Sidoravicius10}
V.~Sidoravicius and A.-S. Sznitman.
\newblock Connectivity bounds for the vacant set of random interlacements.
\newblock {\em Ann. Inst. Henri Poincar\'e Probab. Stat.}, 46(4):976--990,
  2010.

\bibitem[Szn10]{Sznitman07}
A.-S. Sznitman.
\newblock Vacant set of random interlacements and percolation.
\newblock {\em Ann. of Math. (2)}, 171(3):2039--2087, 2010.

\bibitem[TW11]{Tykesson10}
J.~Tykesson and D.~Windisch.
\newblock Percolation in the vacant set of poisson cylinders.
\newblock {\em Probability Theory and Related Fields}, pages 1--27, 2011.

\bibitem[Win00]{Winkler00}
P.~Winkler.
\newblock Dependent percolation and colliding random walks.
\newblock {\em Random Structures Algorithms}, 16(1):58--84, 2000.

\end{thebibliography}

\end{document}